\documentclass[11pt]{amsart}
\usepackage[utf8]{inputenc}
\usepackage[pdftex]{graphicx}
\usepackage{amssymb}
\usepackage{amsmath}
\usepackage{amsfonts}
\usepackage{amsthm}
\usepackage{bm}
\usepackage{mathrsfs}
\usepackage{appendix}
\usepackage{enumerate}
\usepackage[left=3cm, right=3cm, bottom=3.4cm]{geometry}
\usepackage[stretch=30,shrink=30]{microtype}
\usepackage[doi=false,isbn=false,url=false,eprint=false,backref=true,backend=bibtex,giveninits=true,sorting=nyt]{biblatex}
\renewbibmacro{in:}{}
\DeclareFieldFormat[article]{title}{#1}
\usepackage{stmaryrd}
\usepackage[normalem]{ulem} 

\usepackage{xcolor}
\definecolor{antiquefuchsia}{rgb}{0.57, 0.36, 0.51}
\definecolor{azure}{rgb}{0.0, 0.5, 1.0}

\usepackage[colorlinks = true, linkcolor = black, citecolor = antiquefuchsia]{hyperref}
\usepackage[capitalise,nameinlink]{cleveref}
\crefformat{equation}{#2(#1)#3}
\crefrangeformat{equation}{#3(#1)#4 to~#5(#2)#6}
\crefmultiformat{equation}{#2(#1)#3}{ and~#2(#1)#3}{, #2(#1)#3}{ and~#2(#1)#3}

\makeatletter
\def\th@plain{%
	\thm@notefont{}% same as heading font
	\itshape % body font
}
\def\th@definition{%
	\thm@notefont{}% same as heading font
	\normalfont % body font
}
\makeatother

\numberwithin{equation}{section}

\newtheorem{theorem}{Theorem}[section]
\newtheorem{lemma}[theorem]{Lemma}
\newtheorem{proposition}[theorem]{Proposition}
\newtheorem{corollary}[theorem]{Corollary}

\newtheorem{question}[theorem]{Question}

\theoremstyle{definition}

\theoremstyle{remark}
\newtheorem{remark}[theorem]{Remark}

\newcommand{\de}{\partial}

\newcommand{\R}{\mathbb{R}}

\newcommand{\mres}{\mathbin{\vrule height 1.6ex depth 0pt width
0.13ex\vrule height 0.13ex depth 0pt width 1.3ex}}
\newcommand{\ang}[1]{\langle #1 \rangle}
\renewcommand{\bar}{\overline}

\DeclareMathOperator{\spt}{spt}
 
\DeclareMathOperator{\dist}{dist}

\DeclareMathOperator{\vol}{vol}

\DeclareMathOperator{\tr}{tr}

\DeclareMathOperator{\loc}{loc}

\DeclareMathOperator{\U}{U}
\DeclareMathOperator{\SU}{SU}
\DeclareMathOperator{\Sp}{Sp}

\DeclareMathOperator{\Ad}{Ad}
\DeclareMathOperator{\Lip}{Lip}

\newcommand{\dvol}{d\mathrm{vol}}
\newcommand{\bfi}{\mathbf{i}}
\newcommand{\bfj}{\mathbf{j}}
\newcommand{\bfk}{\mathbf{k}}

\newcommand{\mz}{\frac{1}{2}}

\AtBeginDocument{%
   \def\MR#1{}
}

\begin{document}
\title[Nonabelian Yang--Mills--Higgs and Plateau's problem in codimension three]{Nonabelian Yang--Mills--Higgs and \\ Plateau's problem in codimension three}

\author[Davide Parise]{Davide Parise}
\address{University of California San Diego, Department of Mathematics, 9500 Gilman Drive \#0112, La Jolla, CA 92093-0112, United States of America}
\email{dparise@ucsd.edu}

\author[Alessandro Pigati]{Alessandro Pigati}
\address{Bocconi University, Department of Decision Sciences, Via Guglielmo R\"ontgen 1, 20136
Milano, Italy}
\email{alessandro.pigati@unibocconi.it}

\author[Daniel Stern]{Daniel Stern}
\address{Cornell University, Department of Mathematics, 310 Malott Hall, Ithaca, NY 14853}
\email{daniel.stern@cornell.edu}

\date{\today}

\begin{abstract}
       We investigate the asymptotic behavior of the $\SU(2)$-Yang--Mills--Higgs energy $E(\Phi,A)=\int_M|d_A\Phi|^2+|F_A|^2$ in the large mass limit, proving convergence to the codimension-three area functional in the sense of De Giorgi's $\Gamma$-convergence. More precisely, for 
       a compact manifold with boundary $M$ and
       any family of pairs $\Phi_m\in\Omega^0(M;\mathfrak{su}(2))$ and $A_m\in \Omega^1(M;\mathfrak{su}(2))$ indexed by a \emph{mass}
       parameter $m\to\infty$, satisfying
       $$E(\Phi_m,A_m)\leq Cm\quad\text{and}\quad\lim_{m\to\infty}\frac{1}{m}\int_M(m-|\Phi_m|)^2=0,$$
       we prove that the $(n-3)$-currents dual to $\frac{1}{2\pi m}\tr(d_{A_m}\Phi_m\wedge F_{A_m})$ converge subsequentially to a relative integral $(n-3)$-cycle $T$ of mass
       \begin{equation}\label{massineq}
       \mathbb{M}(T)\leq \liminf_{m\to\infty}\frac{1}{4\pi m}E(\Phi_m,A_m),
       \end{equation}
       and show conversely that any integral $(n-3)$-current $T$ with $[T]=0\in H_{n-3}(M,\partial M;\mathbb{Z})$ admits such an approximation, with equality in \eqref{massineq}. In the special case of pairs $(\Phi_m,A_m)$ satisfying the generalized monopole equation $*d_{A_m}\Phi_m=F_{A_m}\wedge \Theta$ for a calibration form $\Theta\in \Omega^{n-3}(M)$, we deduce that the limit $\nu=\lim_{m\to\infty}\frac{1}{2\pi m}|d_{A_m}\Phi_m|^2$ of the Dirichlet energy measures satisfies $\nu\leq |T|$, with equality if and only if $T$ is calibrated by $\Theta$, giving evidence for predictions of Donaldson--Segal in the settings of $G_2$-manifolds and Calabi--Yau $3$-folds.
       
\end{abstract}

\maketitle 
\tableofcontents 

\section{Introduction}

Following breakthroughs by Taubes \cite{TaubesSWGR}, Tian \cite{TianGTCG} and others in the late 1990s, interactions between gauge theory and submanifold geometry have played a central role in the development of gauge theory over the last quarter century. In some settings, as in \cite{TianGTCG}, minimal submanifolds arise as possible degenerations at the boundary of some gauge-theoretic moduli space, but in others, as in \cite{TaubesSWGR}, one finds a robust dictionary between solutions of certain gauge-theoretic PDEs and distinguished submanifolds by passing to adiabatic limits.

On the PDE side, dictionaries of the latter kind have attracted considerable attention since the 1970s, when De Giorgi's school began to explore a correspondence between semilinear scalar equations and minimal hypersurfaces \cite{ModicaMortola,Modicabis}, which in recent years has been used to obtain some striking results in the min-max theory for minimal surfaces and geodesics \cite{Guaraco, GasparGuaraco, ChoMan1, ChoMan2}. Since the 1990s, similar relationships have been found between minimal submanifolds of codimension two and elliptic systems related to the Ginzburg--Landau model of superconductivity \cite{RivU1, LinRiviere, JSgamma, BethuelBrezisOrlandi}, with a particularly satisfying dictionary in the case of the self-dual $\U(1)$-Higgs equations \cite{PS, PPSgamma}. We note in particular that the codimension-two concentration phenomena in \cite{PS,PPSgamma} and \cite{TaubesSWGR} are closely related, with solutions of either the second order self-dual $\U(1)$-Higgs equations or the perturbed Seiberg--Witten equations resembling solutions of the classical vortex equations on $\mathbb{R}^2$ in normal planes to the concentration set at generic points. 

In the search for analogous phenomena in codimension three, it is natural to replace the two-dimensional vortex equations with their nonabelian cousin on $3$-manifolds: the Bogomolnyi monopole equation, with structure group $\SU(2)$. For the trivial $\SU(2)$-bundle over $\mathbb{R}^3$, a \emph{monopole} is a pair $(\Phi,A)$ consisting of a section $\Phi$ of the adjoint bundle $\mathfrak{su}(2)\times \mathbb{R}^3$ and a connection $A\in \Omega^1(\mathbb{R}^3,\mathfrak{su}(2))$ solving the Bogomolnyi equation
\begin{equation}\label{mono}
*d_A\Phi=\pm F_A.
\end{equation}
Like the vortex equations, the Bogomolnyi equation arises as a special case of the instanton equations on four-manifolds, describing translation-invariant instantons of the form $\Phi\, dt+A$ on $\mathbb{R}\times \mathbb{R}^3$. Variationally, monopoles are minimizers of the Yang-Mills-Higgs energy
$$E(\Phi,A):=\int_{\mathbb{R}^3}|d_A\Phi|^2+|F_A|^2,$$
and in fact are the only stable, finite-energy solutions of the Euler--Lagrange equations
\begin{equation}\label{ymh.eq}
    d_A^*d_A\Phi=0, \quad d_A^*F_A+[\Phi,d_A\Phi]=0
\end{equation}
for the Yang--Mills--Higgs energy on $\mathbb{R}^3$ \cite{JaffeTaubes, TaubesStable}. As discussed in \cite{JaffeTaubes,FadelAC}, the energy of any finite-energy monopole $(\Phi,A)$ on $\mathbb{R}^3$ is determined by the product 
$$E(\Phi,A)=4\pi m |k|$$
of the (well-defined) limit
$$m=\lim_{|x|\to\infty}|\Phi(x)|$$
known as the \emph{mass}, and the \emph{charge}
$$k=\lim_{r\to\infty}\deg(\Phi/|\Phi|, S^2_r(0))\in \mathbb{Z}.$$
In particular, after normalizing by the mass, the energy $\frac{1}{4\pi m}E(\Phi,A)$ is quantized, and the energy measures 
$$\frac{1}{4\pi m}(|d_A\Phi|^2+|F_A|^2) \,\dvol_g=\pm\frac{1}{2\pi m}\tr(d_A\Phi\wedge F_A)\,\dvol_g$$
tend to concentrate at points in the \emph{large mass limit} $m\to\infty$; see \cite{FadelOliveira} for a detailed analysis of the large-mass asymptotics of monopoles on general asymptotically conical $3$-manifolds. 

In higher dimensions, analogs of the monopole equation \eqref{mono} arise naturally in certain manifolds with special holonomy, most notably in (6-dimensional) Calabi--Yau $3$-folds and (7-dimensional) $G_2$-manifolds, where solutions are again critical points for the Yang--Mills--Higgs energy
$$E(\Phi,A)=\int_M|d_A\Phi|^2+|F_A|^2$$
with respect to compact variations (see, e.g., \cite[Section 1.3]{OliveiraThesis}). In \cite{DonaldsonSegal}, Donaldson and Segal suggest that suitable counts of these generalized monopoles on noncompact $G_2$-manifolds or Calabi--Yau $3$-folds could provide meaningful enumerative invariants, similar to classic gauge-theoretic invariants of lower-dimensional manifolds. Moreover, they conjecture that in the \emph{large mass limit} these invariants can be identified with certain counts of special Lagrangians in Calabi--Yau $3$-folds or coassociatives in $G_2$-manifolds (weighted by a count of \emph{Fueter sections} describing first-order asymptotics for monopoles near the concentration set), providing a kind of nonabelian, codimension-three counterpart to the correspondence between Seiberg--Witten and Gromov invariants in symplectic four-manifolds. Starting with Oliveira's thesis \cite{OliveiraThesis}, the last decade has seen some interesting progress on the Donaldson--Segal program: see, for instance, \cite{OliveiraJGP, Stein, FadelNagyOliveira, EsfahaniThesis} and references therein. Nonetheless, at the moment the central conjectures remain widely open.

Motivated in part by the Donaldson--Segal picture and drawing inspiration from the dictionary between the self-dual $\U(1)$-Higgs equations and minimal submanifolds of codimension two, we turn now to the following question.

\begin{question}
    Is there a robust correspondence between minimal submanifolds of codimension three and solutions of the $\SU(2)$-Yang--Mills--Higgs equations in the large mass limit? More broadly, is there a correspondence between the Yang--Mills--Higgs energies and the codimension-three area functional, at the level of $\Gamma$-convergence or convergence of gradient flows, in suitable adiabatic limits?
\end{question}

What follows is the first part of a positive answer, proving that \emph{the $\SU(2)$-Yang--Mills--Higgs energies converge in a natural sense to the mass functional for $(n-3)$-cycles in adiabatic limits}. Morally speaking, the precise result stated below may be compared with the results of Modica--Mortola for the Allen--Cahn functionals \cite{ModicaMortola}, Jerrard--Soner and Alberti--Baldo--Orlandi for the ungauged Ginzburg--Landau equations \cite{JSgamma, ABO}, and the authors for the self-dual $\U(1)$-Higgs energies \cite{PPSgamma}, but at a technical level both the statement and its proof are necessarily rather different from all of these. 

To illustrate the geometric content of the $\Gamma$-convergence result, we show that area-minimizing $(n-3)$-cycles can be approximated locally by minimizers (or minimizing sequences) for the $\SU(2)$-Yang--Mills--Higgs energy; in other words, Plateau's problem in codimension three can be solved (however impractically) by gauge-theoretic means. As a further consequence of our results, we show that, for generalized monopoles of the kind relevant to the Donaldson--Segal program, the Dirichlet component $|d_A\Phi|^2$ of the energy measure always concentrates in the large mass limit along the support of an integral $(n-3)$-cycle; moreover, this cycle is calibrated, e.g., coassociative in the $G_2$ setting, whenever a certain integral balancing condition between the Dirichlet and Yang--Mills components of the energy measure holds near the concentration set.

\subsection{\for{toc}{$\Gamma$-convergence of the rescaled Yang--Mills--Higgs energies}\except{toc}{$\bm{\Gamma}$-convergence of the rescaled Yang--Mills--Higgs energies}}\hfill\\
Let $(M^n,g)$ be a compact Riemannian manifold, possibly with boundary, and let $P\to M$ be a principal $\SU(2)$-bundle over $M$, with associated adjoint bundle $\mathcal{E}\to M$. Since we are primarily concerned with local phenomena in this paper, let us assume for simplicity that $P=\SU(2)\times M$ is the trivial bundle, so that $\mathcal{E}=\mathfrak{su}(2)\times M$. Fixing a trivialization, we then identify connections on $P$ with $\mathfrak{su}(2)$-valued one-forms $A\in\Omega^1(M;\mathfrak{su}(2)),$ acting on sections $\Phi:M\to \mathfrak{su}(2)$ of the adjoint bundle by
$$d_A\Phi=d\Phi+[A,\Phi],$$
with curvature given by
$$F_A=dA+\frac12[A\wedge A].$$
Throughout the paper, we identify $\SU(2)$ with the $3$-sphere of unit quaternions and $\mathfrak{su}(2)$ with the imaginary quaternions, equipped with the standard inner product with respect to which $|\bf i|=|\bf j|=|\bf k|=1$. See Section \ref{prelim} below for further discussion of our notation and conventions, and comparison with other common conventions in the literature.

For each $\epsilon>0$, we define the $\epsilon$-Yang--Mills--Higgs energy of a pair $(\Phi,A)$ by
$$E_{\epsilon}(\Phi,A):=\int_M \frac{1}{\epsilon}|d_A\Phi|^2+\epsilon|F_A|^2.$$
Note that the energies $E_{\epsilon}$ are related to the standard ($\epsilon=1$) Yang-Mills-Higgs energies via two different scalings: first, with respect to the rescaled metric $g_{\epsilon}=\epsilon^{-2}g$, we see that
$$E_1^{g_{\epsilon}}(\Phi,A)=\epsilon^{3-n}E_{\epsilon}^g(\Phi,A);$$
on the other hand, fixing the metric $g$ but rescaling the section by taking $\Psi=\epsilon^{-1}\Phi$, we see that
$$E_{\epsilon}(\Phi,A)=\epsilon E_1(\Psi,A).$$
The latter observation is particularly relevant when recasting the $\epsilon\to 0$ asymptotics of the energies $E_{\epsilon}$ in the language of \emph{large mass limits} often found in the gauge theory literature.

To any section $\Phi:M\to \mathfrak{su}(2)$ and connection $A\in \Omega^1(M;\mathfrak{su}(2))$, we associate a real-valued $3$-form
$$Z(\Phi,A):=2\Re(d_A\bar{\Phi}\wedge F_A)\in\Omega^3(M),$$
where we denote by $\Re(\cdot)$ the real part of a quaternion-valued form. For pairs $(\Phi,A)$ on $\mathbb{R}^3$ with suitable decay at infinity and $|\Phi(x)|\to 1$ as $|x|\to\infty$, note that this is precisely the three-form whose integral recovers $4\pi$ times the integer charge of $(\Phi,A)$ \cite{JaffeTaubes, FadelAC}. In any manifold, a simple computation (see Proposition \ref{trivial.ub} below) gives the pointwise bound
$$|Z(\Phi,A)|\leq \frac{1}{\epsilon}|d_A\Phi|^2+\epsilon |F_A|^2,$$
and in particular
$$\int_M|Z(\Phi,A)|\leq E_{\epsilon}(\Phi,A).$$
On a $3$-manifold, equality holds when $(\Phi,A)$ satisfies the rescaled Bogolmonyi equation $*d_A\Phi=\pm\epsilon F_A$, in which case $Z(\Phi,A)=\pm (\epsilon^{-1}|d_A\Phi|^2+\epsilon|F_A|^2)$ simply recovers the Yang--Mills--Higgs energy density.

Our first main result, giving the $\liminf$ part of the $\Gamma$-convergence statement, shows that for a family of pairs $(\Phi_{\epsilon},A_{\epsilon})$ with bounded energy $E_{\epsilon}(\Phi_{\epsilon},A_{\epsilon})\leq C$ and $\|1-|\Phi_{\epsilon}|\|_{L^2}\to 0$ sufficiently fast as $\epsilon\to 0$, the $3$-forms $Z(\Phi_{\epsilon},A_{\epsilon})$ converge to an integral $(n-3)$-cycle.  To make this convergence precise, note that $Z(\Phi,A)$ naturally defines an $(n-3)$-current via 
$$\Omega^{n-3}(M)\ni \alpha\mapsto \int_MZ(\Phi,A)\wedge \alpha.$$

\begin{theorem} \label{thm: main theorem}
    Given a sequence $\epsilon_j\to 0$, a sequence of smooth $\SU(2)$-connections $A_j$ and sections $\Phi_j\in \Gamma(\mathcal{E})$ such that
    $$\liminf_{j\to\infty}E_{\epsilon_j}(\Phi_j,A_j)<\infty$$
    and
    \begin{equation}\label{1.fast}
    \lim_{j\to\infty}\int_M\frac{(1-|\Phi_j|)^2}{\epsilon_j}=0,
    \end{equation}
    there exist an $(n-3)$-current $T$, restricting to an integral cycle in the interior of $M$, and a measure $\mu\in C^0(M)^*$ such that, along a subsequence,
    $$Z(\Phi_j,A_j)\rightharpoonup^* 4\pi T$$
    as $(n-3)$-currents, 
    $$(\epsilon_j^{-1}|d_{A_j}\Phi_j|^2+\epsilon_j|F_{A_j}|^2)\,\dvol_g\rightharpoonup^* \mu$$
    in $C^0(M)^*$, and the weight measure $|T|=\theta\mathcal{H}^{n-3}\mres \spt(T)$ satisfies
    \begin{equation}\label{ener.lim}
    4\pi |T|\leq \mu;
    \end{equation}
    in particular,    
    $$4\pi \mathbb{M}(T)\leq \liminf_{j\to\infty} E_{\epsilon_j}(\Phi_j,A_j).$$
\end{theorem}

\begin{remark}
    Note that Theorem \ref{thm: main theorem} only establishes the integrality of the limit current $T=\lim_{j\to\infty}\frac{1}{4\pi}Z(\Phi_j,A_j)$ in the \emph{interior} of $M$; indeed, without additional assumptions, it is not difficult to construct examples where, e.g., $T$ is supported in $\partial M$ and does not have the structure of an integral $(n-3)$-current. On the other hand, as we will see in Section \ref{plat.sec} below, it is possible to guarantee that $T$ is an integral $(n-3)$-current on the full manifold with boundary $M$ by imposing additional natural assumptions on the boundary data $\iota_{\partial M}^*(\Phi_{\epsilon},A_{\epsilon})$.
\end{remark}

The hypothesis \eqref{1.fast}, providing a local, integral version of the `large mass' assumption, deserves some additional comment. First, note that for any sequence $(\Phi_j,A_j)$ with $E_{\epsilon_j}(\Phi_j,A_j)\leq C$ we have
$$\int_M |d|\Phi_j||^2\leq \int_M |d_{A_j}\Phi_j|^2\leq C\epsilon_j$$
and, setting $c_j:=\frac{1}{|M|}\int_M|\Phi_j|$, the Poincar\'e inequality on $M$ gives
$$\int_M |c_j-|\Phi_j||^2\leq C\epsilon_j.$$
Under the mild assumption that $\liminf_{j\to\infty}c_j> 0$, we can then normalize by $\tilde\Phi_j:=c_j^{-1}\Phi_j$ to get a new bounded-energy sequence for which
$$\int_M \frac{|1-|\tilde\Phi_j||^2}{\epsilon_j}=O(1)$$
as $j\to\infty$, and the assumption \eqref{1.fast} simply strengthens this estimate from $O(1)$ to $o(1)$. To see that the improvement from $O(1)$ to $o(1)$ is necessary, fix a function $f\in C^{\infty}(M)$ and a one-form $\alpha\in \Omega^1(M)$, and consider the family
$$\Phi_{\epsilon}=(1+\sqrt{\epsilon} f){\bf i},\text{ }A_{\epsilon}=\bf i \frac{\alpha}{\sqrt{\epsilon}};$$
it is easy to see that $E_{\epsilon}(\Phi_{\epsilon},A_{\epsilon})\leq C$ and $\int_M (1-|\Phi_{\epsilon}|)^2\leq C\epsilon$, but $Z(\Phi_{\epsilon},A_{\epsilon})=2df\wedge d\alpha$ does not concentrate to an $(n-3)$-cycle as $\epsilon \to 0$. On the other hand, for critical points of $E_{\epsilon}$, it is straightforward to obtain estimates somewhat stronger than \eqref{1.fast} in many natural settings, as in the case of Theorem \ref{plateau.thm} below.

Conversely, we show that every integral $(n-3)$-cycle can be approximated by a sequence of pairs satisfying the hypotheses of Theorem \ref{thm: main theorem} such that equality holds in \eqref{ener.lim}, giving the $\limsup$ part of the $\Gamma$-convergence result.

\begin{theorem}\label{thm: recovery}
For every integral $(n-3)$-current $T$%\in \mathcal{I}_{n-3}(M;\mathbb{Z})
such that $\spt(\partial T)\subset \partial M$ and $[T]=0\in H_{n-3}(M,\partial M; \mathbb{Z})$, there exists a family of smooth pairs $(\Phi_{\epsilon},A_{\epsilon})$ satisfying $|\Phi_\epsilon|\le1$,
$$Z(\Phi_{\epsilon},A_{\epsilon})\rightharpoonup^* 4\pi T$$
as $(n-3)$-currents, 
$$\lim_{\epsilon\to 0}E_{\epsilon}(\Phi_{\epsilon},A_{\epsilon})=4\pi \mathbb{M}(T),$$
and
$$\int_M\frac{(1-|\Phi_{\epsilon}|)^2}{\epsilon}\leq C\mathbb{M}(T)\epsilon\to0.$$
%as $\epsilon\to 0$.
\end{theorem}

The proof of Theorem \ref{thm: recovery} is similar in spirit to that of \cite[Theorem 1.2(ii)]{PPSgamma} in the $\U(1)$-Higgs setting, with rescalings of the standard charge one Bogomolnyi--Prasad-Sommerfield monopole in $\mathbb{R}^3$ playing a role analogous to that of the degree one vortex in \cite{PPSgamma}. 

\begin{remark}
Given a function $\delta: (0,1)\to (0,1)$ satisfying
$$\lim_{\epsilon\to 0}\frac{\delta(\epsilon)}{\epsilon}=\lim_{\epsilon\to 0}\frac{\epsilon^2}{\delta(\epsilon)}=0$$
(e.g., $\delta(\epsilon)=\epsilon^{3/2}$), we can consider instead the perturbed Yang--Mills--Higgs functionals
$$\tilde{E}^{\delta(\epsilon)}_{\epsilon}(\Phi,A):=E_{\epsilon}(\Phi,A)+\frac{1}{\delta(\epsilon)}\int_M(1-|\Phi|^2)^2.$$
It is then easy to see that any family $(\Phi_{\epsilon},A_{\epsilon})$ with 
$$\tilde{E}_{\epsilon}^{\delta(\epsilon)}(\Phi_{\epsilon},A_{\epsilon})\leq C$$
automatically satisfies the hypotheses of Theorem \ref{thm: main theorem}, while the recovery sequence $(\Phi_{\epsilon},A_{\epsilon})$ described in Theorem \ref{thm: recovery} satisfies
$$\lim_{\epsilon\to 0}\tilde{E}_{\epsilon}^{\delta(\epsilon)}(\Phi_{\epsilon},A_{\epsilon})=\lim_{\epsilon\to 0}E_{\epsilon}(\Phi_{\epsilon},A_{\epsilon}).$$
In particular, Theorem \ref{thm: main theorem} and Theorem \ref{thm: recovery} together show that the functionals $\tilde{E}_{\epsilon}^{\delta(\epsilon)}$ $\Gamma$-converge to the $(n-3)$-mass functional in the usual sense, without any additional hypotheses.

From a gauge-theoretic perspective, these perturbed functionals are not quite as interesting as the canonical Yang--Mills--Higgs energies $E_{\epsilon}$, losing their connection to monopoles and instantons. On the the other hand, from a variational perspective, it may be of interest to note that the perturbed functionals $\tilde{E}_{\epsilon}^{\delta(\epsilon)}$ are expected to admit nontrivial, finite-energy critical points on compact manifolds \emph{without} boundary
(e.g., by a variant of the construction in \cite[Section 2]{SternJDG}
or \cite[Section 7]{PS}), in contrast to the usual energies $E_{\epsilon}$, whose finite-energy critical points on closed manifolds consist only of parallel sections and Yang--Mills connections.
\end{remark}

As a concrete illustration of the convergence phenomenon captured by Theorems \ref{thm: main theorem} and \ref{thm: recovery}, we observe next that one can solve area-minimization problems in codimension three via variational methods for the $\SU(2)$-Yang--Mills--Higgs energies.

\begin{theorem}\label{plateau.thm}
 Let $\Gamma^{n-4}\subset \partial M$ be any smooth $(n-4)$-dimensional submanifold of $\partial M$ such that $[\Gamma]=0\in H_{n-4}(M;\mathbb{Z}).$ Then there exists a family of smooth pairs $\Psi_{\epsilon}:\partial M\to \mathfrak{su}(2)$ and $B_{\epsilon}\in \Omega^1(\partial M;\mathfrak{su}(2))$ such that the following holds: letting
 $$\alpha_{\epsilon}(\Psi_\epsilon, B_\epsilon):=\inf\{E_{\epsilon}(\Phi,A)\mid \iota_{\partial M}^*(\Phi,A)=(\Psi_{\epsilon},B_{\epsilon})\},$$
 we have
 $$\lim_{\epsilon\to 0}\alpha_{\epsilon}(\Psi_\epsilon, B_\epsilon) =4\pi\inf\{\mathbb{M}(T)\mid T
 %\in \mathcal{I}_{n-3}(M;\mathbb{Z})
 \text{ integral with }\partial T=\Gamma\}$$
 and, for any sequence of pairs $(\Phi_{\epsilon},A_{\epsilon})$ satisfying $$\iota_{\partial M}^*(\Phi_{\epsilon},A_{\epsilon})=(\Psi_{\epsilon},B_{\epsilon}) \quad \text{and} \quad E_{\epsilon}(\Phi_{\epsilon},A_{\epsilon})\leq \alpha_{\epsilon}(\Psi_\epsilon, B_\epsilon) + o(1),$$ we have
 $$Z(\Phi_{\epsilon},A_{\epsilon})\rightharpoonup^* 4\pi T,$$
 where the integral $(n-3)$-current $T$ is a mass-minimizing extension of $\Gamma$.
\end{theorem}

Roughly speaking, Theorem \ref{plateau.thm} shows that pairs $(\Phi_{\epsilon},A_{\epsilon})$ minimizing $E_{\epsilon}$ with respect to boundary data concentrating along a prescribed codimension-four cycle $\Gamma$ converge to an $(n-3)$-current $T$ minimizing area among all currents with boundary $\partial T=\Gamma$. However, there is a subtle technical point which prevents us from stating the theorem in this form: in high dimensions, it is not clear a priori in what sense a minimizing pair \emph{exists}, or what kind of partial regularity such a pair should enjoy. This is very closely related to the existence and partial regularity problem for minimizers of the Yang-Mills energy in supercritical dimension; we refer the reader to the recent paper of Caniato--Rivi\`ere for further discussion of these issues and recent progress in dimension five \cite{CaniatoRiviere}.

\subsection{Large mass limits of generalized monopoles and general critical points}\hfill\\
Again, let $P=\SU(2)\times M$ be the trivial $\SU(2)$-bundle over a compact Riemannian manifold $(M^n,g)$ with boundary $\partial M$.
%perhaps arising as a compact subset of a complete, noncompact manifold.
Following \cite[Section 1.3]{OliveiraThesis}, suppose now that $M$ carries a calibration $(n-3)$-form, i.e., a closed $(n-3)$-form $\Theta$ with comass one at every point, and call a pair $(\Phi,A)$ a \emph{$\Theta$-monopole} if
\begin{equation}\label{gen.mono}
    *d_A\Phi=F_A\wedge \Theta,
\end{equation}
or more generally, after rescaling $\Phi$,
\begin{equation}\label{eps.mono}
*d_A\Phi=\epsilon F_A\wedge \Theta
\end{equation}
for some $\epsilon>0$. 

Of particular interest in light of the Donaldson--Segal program are the cases of $G_2$-monopoles, where $(M^7,g)$ is a $G_2$-manifold and $\Theta$ is the coassociative $4$-form, and Calabi--Yau monopoles on a Calabi--Yau $3$-fold $(M^6,g)$, where $\Theta$ is the real part of the holomorphic volume form and $F_A\wedge \omega^2=0$ for the K\"ahler form $\omega$ (see, e.g., \cite[Chapters 3--4]{OliveiraThesis}). In \cite{DonaldsonSegal}, Donaldson and Segal suggest that $G_2$-monopoles on a complete, noncompact $G_2$-manifold $M$ with a well-defined mass at infinity
$$m=\lim_{|x|\to \infty}|\Phi(x)|$$
should concentrate as $m\to\infty$ along coassociative submanifolds,
which are four-dimensional cycles calibrated by $\Theta$, and predict an analogous large-mass convergence of Calabi--Yau monopoles to special Lagrangian submanifolds in the Calabi--Yau setting, as has been confirmed by Oliveira in some special cases \cite{OliveiraThesis}.

Equivalent to the study of large-mass asymptotics for the equation \eqref{gen.mono}, one can consider instead the $\epsilon\to 0$ behavior of solutions to \eqref{eps.mono} satisfying $|\Phi|\to 1$ at infinity. Restricting attention to compact subsets of complete $G_2$-manifolds or Calabi--Yau $3$-folds, the conjectural picture of Donaldson and Segal suggests the following local question.

\begin{question}
    In the setting of Theorem \ref{thm: main theorem}, suppose in addition that the pairs $(\Phi_j,A_j)$ solve the $\epsilon_j$-monopole equation \eqref{eps.mono}. Under what conditions does it follow that the limiting $(n-3)$-current $T$ is calibrated by $\Theta$? In particular, does this always hold for $G_2$-monopoles or Calabi--Yau monopoles?
\end{question}

As an immediate application of Theorem \ref{thm: main theorem}, we have the following.

\begin{corollary}\label{mono.lim}
    In addition to the assumptions of Theorem \ref{thm: main theorem}, suppose that $(\Phi_j,A_j)$ satisfies the generalized monopole equation
    \begin{equation}\label{quasi.mono}
        *d_{A_j}\Phi_j=\epsilon_j F_{A_j}\wedge \Theta
    \end{equation}
    with respect to a fixed calibration form $\Theta\in \Omega^{n-3}(M)$. Writing
    $$\nu:=\lim_{j\to\infty}(\epsilon_j^{-1}|d_{A_j}\Phi_j|^2+\epsilon_j|F_{A_j}\wedge \Theta|^2) \,\dvol_g = \lim_{j\to\infty} 2\epsilon_j^{-1}|d_{A_j}\Phi_j|^2 \,\dvol_g,$$
    the weight measure $|T|$ of the limiting integral $(n-3)$-current $T$ satisfies
    \begin{equation}\label{ener.below}
    \nu\leq 4\pi |T|,
    \end{equation}
    with equality on $M\setminus\de M$ if and only if $T$ is calibrated by $\Theta$. In particular, in the interior of $M$, $\nu$ is always $(n-3)$-rectifiable, and if $\nu=\mu$ then $\frac{1}{4\pi}\mu=|T|=\frac{1}{4\pi}\nu$ is the weight measure of a calibrated $(n-3)$-current.
\end{corollary}

\begin{proof}
The bound \eqref{ener.below} is a straightforward consequence of Theorem \ref{thm: main theorem}, together with the observation that
\begin{align*}
    Z(A,\Phi)\wedge \Theta&=2 \Re (d_A\bar{\Phi}\wedge F_A\wedge \Theta)\\
    &=2\epsilon_j^{-1}\Re(d_A\bar{\Phi}\wedge *d_A\Phi)\\
    &=2\epsilon_j^{-1}|d_A\Phi|^2 %\,\dvol_g=(\epsilon_j^{-1}|d_A\Phi|^2+\epsilon_j|F_A\wedge \Theta|^2) \,\dvol_g
\end{align*}
for any solution $(\Phi,A)$ of \eqref{quasi.mono}. In particular, passing to limits as $j\to\infty$ gives
$$\int_M \chi \,d\nu=4\pi\langle T,\chi \Theta\rangle$$
for any test function $0 \leq\chi\in C^0(M)$, and since $\Theta$ is a calibration, we know that 
$$\int_M \chi\, d\nu= 4 \pi \langle T,\chi\Theta\rangle\leq 4 \pi \int_M \chi \,d|T|,$$
with equality for all $\chi\in C^0_c(M\setminus\de M)$ if and only if $T$ is calibrated by $\Theta$ (on $M\setminus\de M$). Moreover, since $4\pi|T|\leq \mu$ by Theorem \ref{thm: main theorem}, we deduce that if $\mu=\nu$ then $\mu=4\pi|T|=\nu$, and $T$ must be calibrated.
\end{proof}

For $G_2$-monopoles in asymptotically conical $G_2$ manifolds, the $4$-rectifiability of the `intermediate energy' measure $\nu$ was previously proved in \cite{FadelThesis} via PDE methods under slightly different hypotheses, including the additional assumption that $|F_{A_j}|^2-|F_{A_j}\wedge \Theta|^2$ is uniformly bounded in $L^{\infty}$. In fact, if we impose such a bound alongside the hypotheses of Corollary \ref{mono.lim}, it follows immediately that 
$$\mu-\nu=\lim_{j\to\infty}\epsilon_j(|F_{A_j}|^2-|F_{A_j}\wedge \Theta|^2)\,\dvol_g=0,$$
and therefore Corollary \ref{mono.lim} gives that $\frac{1}{4\pi}\nu=|T|=\frac{1}{4\pi}\mu$ is indeed calibrated, and in particular coassociative, in the case of $G_2$-monopoles. However, it is unclear to what extent this $L^{\infty}$ bound on $|F_{A_j}|^2-|F_{A_j}\wedge \Theta|^2$ can be justified in general settings of interest. 

More generally, the question of balancing between the terms $\epsilon^{-1}|d_A\Phi|^2$ and $\epsilon|F_A|^2$, at least in an integral sense at small scales, is central to understanding concentration phenomena for critical points of $E_{\epsilon}$ in the $\epsilon\to 0$ limit. Indeed, standard computations show that, on small balls $B_r(x)\subset M$, critical pairs $(\Phi,A)$ for $E_{\epsilon}$ satisfy
$$\frac{d}{dr}\left(r^{3-n}\int_{B_r(x)}(\epsilon^{-1}|d_A\Phi|^2+\epsilon|F_A|^2)\right) \gtrsim r^{2-n}\int_{B_r(x)}(\epsilon^{-1}|d_A\Phi|^2-\epsilon|F_A|^2).$$
Thus, controlling $\int_{B_r}\epsilon|F_A|^2$ by $\int_{B_r}\epsilon^{-1}|d_A\Phi|^2$ up to smaller terms is essential for upgrading the obvious codimension-four monotonicity of energy to a sharper codimension-three monotonicity in settings where codimension-three concentration is expected. In the Allen--Cahn and $\U(1)$-Higgs settings, analogous balancing results between different components of the energy can be proved in great generality at the level of pointwise estimates \cite{Modica, PS}, but simple examples show that the same cannot be expected for $\SU(2)$-Yang--Mills--Higgs pairs. In particular, the degenerate case where $\Phi=0$ and $A$ is a Yang--Mills connection shows that the codimension-four energy growth is optimal for general critical pairs, without some largeness condition on the Higgs field $\Phi$. 

Even with a largeness condition on $\Phi$, one also has to contend with the `abelian' case where $|\Phi|=1$ and $d_A\Phi=0$.
In this case, the curvature $F_A$ reduces to a real-valued harmonic two-form: indeed,
locally and up to a change of gauge, $\Phi=\mathbf{i}$ and $A=\alpha\mathbf{i}$, so that $F_A=d\alpha\mathbf{i}$ and $d^*d\alpha=0$. In this case $|F_A|^2$ could be large a priori, but is controlled pointwise by the Yang--Mills energy, adding at worst a diffuse component to the limiting energy measure, reminiscent of the situation for the (ungauged) complex Ginzburg--Landau equations described in \cite[Remark 1.2]{SternJDG}. With these caveats in mind, the following seems like the natural \emph{local} question, which we hope to address in future work.

\begin{question}\label{cp.lim}
    Let $P\to M^n$ be the trivial $\SU(2)$-bundle over a compact Rimannian manifold with boundary of dimension $n\geq 4$, and let $(\Phi_{\epsilon},A_{\epsilon})$ be a family of critical pairs for the $\SU(2)$-Yang--Mills--Higgs energies $E_{\epsilon}$ such that $|\Phi_{\epsilon}|\leq 1$, $E_{\epsilon}(\Phi_{\epsilon},A_{\epsilon})\leq C,$ and $\int_M(1-|\Phi_{\epsilon}|)^2$ vanishes sufficiently fast. %(at least $o(\epsilon)$, perhaps $O(\epsilon^2)$) as $\epsilon\to 0$.
    In the interior of $M$, does the limiting energy measure $$\mu=\lim_{\epsilon\to 0}(\epsilon^{-1}|d_{A_{\epsilon}}\Phi_{\epsilon}|^2+\epsilon|F_{A_{\epsilon}}|^2) \,\dvol_g$$
    necessarily decompose as the sum $\mu=|V|+|h|^2\,\dvol_g$ of the weight of a stationary, rectifiable $(n-3)$-varifold $V$ and the energy density of a harmonic $2$-form $h$?
\end{question}

An affirmative answer would certainly provide the tools required to prove that $\nu=4\pi|T|$ in the setting of Corollary \ref{mono.lim} for $G_2$-monopoles or Calabi--Yau monopoles, confirming that $T$ is indeed coassociative or special Lagrangian, respectively. 

Note that in Question \ref{cp.lim} we have not made reference to the problem of integrality, i.e., whether the density of the $(n-3)$-varifold $V$ takes values in $4\pi\mathbb{N}$ $|V|$-almost everywhere. Indeed, integrality cannot be expected in general, in light of Taubes's construction of entire Yang--Mills--Higgs pairs of finite energy on $\mathbb{R}^3$ that are not monopoles, whose energy need not take values in $4\pi \mathbb{N}$ \cite{TaubesNonMinimalI, TaubesNonMinimalII}. On the other hand, these non-monopole pairs on $\mathbb{R}^3$ are necessarily unstable by \cite{TaubesStable}, and their trivial pullback to $\mathbb{R}^3\times \mathbb{R}^k$ must have infinite Morse index for any $k\geq 1$, suggesting the following refinement of Question \ref{cp.lim}.

\begin{question}\label{cp.ind.lim}
    In the situation of Question \ref{cp.lim}, assuming a positive answer, suppose moreover that the pairs $(\Phi_{\epsilon},A_{\epsilon})$ have uniformly bounded Morse index as critical points of $E_{\epsilon}$. Is $\frac{1}{4\pi}V$ then an integral varifold?
\end{question}

\begin{remark}
    In the base dimension $n=3$, one needs to replace the assumption of bounded Morse index with stability; in this case, a positive answer to Question \ref{cp.ind.lim} seems to follow from the analysis of \cite{FadelOliveira} together with \cite{TaubesStable}.
\end{remark}

\subsection*{Acknowledgements} 

The authors thank Gonçalo Oliveira and Daniel Fadel for helpful exchanges about monopoles and comments on an earlier draft of this paper. We also thank Riccardo Caniato and Tristan Rivi\`ere for answering questions about their work \cite{CaniatoRiviere}. D.P. acknowledges the support of the AMS-Simons travel grant. During the completion of this project, A.P. was partially supported by ERC grant 101165368, while D.S. was partially supported by NSF grant DMS 2404992 and the Simons Foundation award MPS-TSM-00007693.

\section{Preliminaries, notation, and some key identities}\label{prelim}

Let $(M^n,g)$ be a Riemannian manifold of dimension $n\geq 3$, and consider the trivial principal $\SU(2)$-bundle $P= \SU(2)\times M$. We denote by $E$ the associated adjoint bundle of $P$, i.e., the vector bundle over $M$ given by
$$E = P \times_{\Ad} \mathfrak{su}(2) \cong \mathfrak{su}(2)\times M,$$
where $\Ad \colon \SU(2) \rightarrow \mathrm{GL}(\mathfrak{su}(2))$ is the usual adjoint representation,
%We fix the metric on $E$ to be the one arising from the inner product $(a, b) \mapsto - 2 \tr(ab)$ on the Lie algebra $\mathfrak{su}(2)$. 
given by $T\mapsto \gamma T\gamma^{-1}$ for every $\gamma\in \SU(2)$, $T\in \mathfrak{su}(2)$. 
Note that, as a manifold and Lie group, $\SU(2)$ is the same as the $3$-sphere $S^3$, viewed as a subset of the quaternions $S^3\subset \mathbb{H}$. For convenience, we will henceforth identify $\SU(2)$ with quaternions of unit norm (sometimes called \textit{versors}), i.e., the compact symplectic group $\Sp(1)$.
%Although not needed in this work, $\SU(2)$ can be further identified with the spin group $\Spin(3)$ as in this case $\Cl^{\mathrm{even}}(3)$ is simply given by the quaternions $\mathbb{H}$. We refer the reader to REF for further details.

Coherently, the Lie algebra $\mathfrak{su}(2)$ %consists of $2 \times 2$ skew-Hermitian matrices with zero trace and it can 
will be identified with the imaginary quaternions %(or pure quaternions)
$\Im(\mathbb{H})$, i.e.,
$$\mathfrak{su}(2)\cong \Im(\mathbb{H}) = \{a\mathbf{i} + b\mathbf{j} + c\mathbf{k} \mid (a,b,c)\in \mathbb{R}^3\},$$
where 
%\begin{equation*}
%    [\mathbf{i}, \mathbf{j}] = 2 \mathbf{k}, \qquad [\mathbf{j}, \mathbf{k}] = 2 \mathbf{i}, \qquad [\mathbf{k}, \mathbf{i}] = 2 \mathbf{j}. 
%\end{equation*}
$$\bfi^2 = \bfj^2 = \bfk^2 = \bfi \bfj \bfk = - 1,\quad\bfi \bfj = - \bfj \bfi = \bfk,
\quad \bfj \bfk = - \bfk \bfj = \bfi,\quad\bfk \bfi = - \bfi \bfk = \bfj.$$
Although we do not need an explicit identification, a standard one is given by the Pauli matrices
$$a\mathbf{i} + b\mathbf{j} + c\mathbf{k} \leftrightarrow
  a \begin{pmatrix}
        i & 0 \\
        0 & -i
    \end{pmatrix}
+ b \begin{pmatrix}
        0 & -1 \\
        1 & 0 
    \end{pmatrix}
+ c \begin{pmatrix}
        0 & i \\
        i & 0
    \end{pmatrix}.$$
%(Note that these are $i$ times the Pauli matrices, up to changing signs in the second one.)
Under this identification, we equip $\mathfrak{su}(2)$ with the Euclidean inner product
$$\langle u,v\rangle:=\Re(\bar{u}v)=-\Re(uv), $$ so that for any imaginary quaternion $u \in \Im(\mathbb{H})$ we have the useful identity
$$u^2=-|u|^2.$$ 
%Furthermore, we have the following
%\begin{equation*}
    %\vert \bfi \vert = 1, \quad \vert \bfj \vert = 1, \quad \vert \bfk \vert = 1. 
%\end{equation*}

\begin{remark}\label{factors.2}
    In many sources, such as the monograph \cite{JaffeTaubes}, the inner product on $\mathfrak{su}(2)$
    is taken to be $\langle a,b\rangle_{JT}:=-2\operatorname{tr}(ab)$, which
    agrees with the one used here multiplied by a factor $4$ (after the previous identification).
    Thus, for instance, the boundary condition $|\Phi|_{JT}\to 1$ at infinity becomes $|\Phi|\to\frac12$ in our convention, and the Yang--Mills--Higgs energies differ by a constant factor. Of course, it is easy to translate between the two conventions, but some care is needed when comparing results in \cite{JaffeTaubes} and the formulations here.
    %Letting $\tilde\Phi:=2\Phi$ and recalling the notation $\operatorname{Tr}=2\operatorname{tr}$ (with $\operatorname{tr}$ the usual trace) used in \cite{Vortices&Monopoles}, the degree  at infinity \cite[eq. (IV.1.10)] correctly reads
    %$$N=\frac{1}{4\pi}\int_{S_\infty^2}4\Re(\bar\Phi d\Phi\wedge d\Phi)
    %=\frac{1}{8\pi}\int_{S_\infty^2}\Re(\bar\tilde\Phi d\tilde\Phi\wedge d\tilde\Phi)
    %=\frac{1}{4\pi}\int_{S_\infty^2}\Phi^*(\omega),$$
    %where $\omega=x^1dx^2\wedge dx^3+x^2 dx^3\wedge dx^1+x^3dx^1\wedge dx^2$
    %restricts to the volume form of $S^2$.
    %Given a critical point $(\Phi,A)$ with this boundary condition, the rescaling
    %$$\tilde\Phi(x):=2\Phi(2\epsilon^{-1}x),\quad\tilde A(x):=2\epsilon^{-1}\Phi(2\epsilon^{-1}x)$$
    %is critical for $E_\epsilon$ and has
    %$$E_\epsilon(\tilde\Phi,\tilde A)=\int_{\R^3}16\epsilon^{-3}\left(|d_A\Phi|^2+|F_A|^2\right)(2\epsilon^{-1}x)\,dx
    %=\int_{\R^3}2\left(|d_A\Phi|^2+|F_A|^2\right).$$
    %Taking into account that $|\cdot|=\frac12|\cdot|_{JT}$, we will have
    %$$E_\epsilon(\tilde\Phi,\tilde A)=\mathcal{A}(\Phi,A)\in 4\pi\Z$$
    %under appropriate integrability conditions, where $\mathcal{A}$ is the action (i.e., energy)
    %considered in \cite{Vortices&Monopoles} (cf. \cite[eq. (IV.1.1)]{Vortices&Monopoles} with $\lambda=0$).
\end{remark}
%Note that denoting the identity $(2 \times 2)$-matrix by $\mathbf{\Id}$, we can identify $\mathbb{R}^4$, $\mathbb{H}$, and the subspace of $\mathrm{GL}(2, \mathbb{C})$ spanned by $\Id$ and the following three matrices 
%\begin{equation} \label{def of i j k}
%    i \sigma_3 = \begin{pmatrix}
%        i & 0 \\
%        0 & -i
%    \end{pmatrix}
%    \qquad 
%    i \sigma_2 = \begin{pmatrix}
%        0 & 1 \\
%        - 1 & 0 
%    \end{pmatrix}
%    \qquad 
%    i \sigma_1 = \begin{pmatrix}
%        0 & i \\
%        i & 0
%    \end{pmatrix}. 
%\end{equation}
%More precisely, we identify $x = (a, b, c, d) \in \mathbb{R}^4$, with $x = a \Id + i (b \sigma_3 + c \sigma_2 + d \sigma_1) \in \mathbb{H} \subset \mathrm{GL}(2, \mathbb{C})$. The matrices $\sigma_1, \sigma_2, \sigma_3$ are the Pauli spin matrices. This is the matrix representation of quaternions. In particular, in this case we would have that $1, \bfi, \bfj, \bfk$ correspond to $\Id, i \sigma_3, i \sigma_2, i\sigma_1$. Note that adopting this representation would need some care as it would introduce extra factors of $2$ in certain expressions.  

\subsection{\for{toc}{Conventions for $\mathfrak{su}(2)$-valued and $\mathbb{H}$-valued forms}\except{toc}{Conventions for $\bm{\mathfrak{su}(2)}$-valued and $\bm{\mathbb{H}}$-valued forms}}\hfill\\
In general, given a vector bundle $E\to M$
and a connection $\nabla$, recall that we can extend it to a differential operator
$$ d_\nabla:\Omega^k(M;E)\to\Omega^{k+1}(M;E) $$
by declaring that $d_\nabla:=\nabla$ on $0$-forms (i.e., sections $\Gamma(E) = \Omega^0(M; E)$) and that, for a pure tensor $s\eta=s\otimes \eta$ (with $s$ a section and $\eta$ a real-valued $k$-form), we have
$$ d_\nabla(s\eta):=\nabla s\wedge\eta+s\,d\eta. $$
A standard computation, which may be taken as the definition of the curvature, shows that
$$ d_\nabla\circ d_\nabla(\omega)=F_\nabla\wedge\omega, $$
where $F_\nabla\in\Omega^2(M;\operatorname{End}(E))$ is the curvature of $\nabla$ and the latter wedge product is defined by the rule $F_\nabla\wedge(s\eta):=(F_\nabla s)\wedge\eta$.

If $E$ has structure group $G$ and is modeled on its Lie algebra $\mathfrak{g}$ (with the adjoint action of $G$ on $\mathfrak{g}$), then we can write $\nabla s=ds+[A,s]$ in a local trivialization,
for a $\mathfrak{g}$-valued one-form $A$.
For general $\mathfrak{g}$-valued forms $\omega,\zeta$, we define $[\omega\wedge\zeta]$ by requiring that,
in a local trivialization,
$$[s\eta\wedge t\theta]=[s,t]\eta\wedge\theta, $$ where $s,t$ take values in $\mathfrak{g}$, while $\eta$ and $\theta$ are differential forms on $M$. If $\omega$ has order $k$ and $\zeta$ has order $\ell$, then
an application of the Jacobi identity gives
$$ d_\nabla[\omega\wedge\zeta]=[(d_\nabla\omega)\wedge\zeta]+(-1)^k[\omega\wedge(d_\nabla\zeta)], $$
but note carefully that
$$[\omega\wedge\zeta]=(-1)^{k\ell+1}[\zeta\wedge\omega].$$
In the same local trivialization, we compute that
$$d_\nabla \omega=d\omega+[A\wedge\omega].$$
In particular, if $s$ is a constant section here (namely, $ds=0$), then
$$F_\nabla s=d_\nabla(d_\nabla s)=d_\nabla[A,s]=[d_\nabla A,s]-[A,d_\nabla s]
=[dA+[A\wedge A],s]-[A\wedge[A,s]].$$
Since we also have $d_\nabla[A,s]=d[A,s]+[A\wedge[A,s]]=[dA,s]+[A\wedge[A,s]]$, we obtain
$$[[A\wedge A],s]-[A\wedge[A,s]]=[A\wedge[A,s]],$$
and thus $[A\wedge[A,s]]=\frac12[[A\wedge A],s]$, giving
$$F_\nabla=dA+\frac12[A\wedge A]\quad\text{i.e.,}\quad F_\nabla s=\left[dA+\frac12[A\wedge A],s\right],$$
hence, $d_\nabla A=F_\nabla+\frac12[A\wedge A]$. Applying $d_{\nabla}$ and using the fact that $[A\wedge[A\wedge A]]=0$,
one recovers the classic Bianchi identity $d_\nabla F_\nabla=0.$

In the setting of this paper, where we deal with a trivial bundle for simplicity, we have sections and forms taking values in the imaginary quaternions $\Im(\mathbb{H})$, so that our $\mathfrak{su}(2)$-valued forms $\Omega^k(M;\mathfrak{su}(2))$ can be regarded as belonging to the larger space $\Omega^k(M;\mathbb{H})$ of quaternion-valued forms. On $\Omega^k(M;\mathbb{H})\times \Omega^l(M;\mathbb{H})$, we can define another natural wedge product by
$$s\eta\wedge t\theta=(st)\eta\wedge\theta\in \Omega^{k+l}(M;\mathbb{H})$$
and, since $[s,t]=st-ts$, it follows that
$$[\omega\wedge\zeta]=\omega\wedge\zeta-(-1)^{kl}\zeta\wedge\omega$$ 
for $\omega\in\Omega^k(M;\mathfrak{su}(2))$ and $\zeta\in \Omega^l(M;\mathfrak{su}(2)).$
 
In particular, $[A\wedge A]=2A\wedge A$, giving
$$F_\nabla=dA+A\wedge A.$$
Also, denoting by $\Re: \Omega^k(M;\mathbb{H})\to \Omega^k(M)$ the real part of a quaternion-valued form, so that $\Re(\beta)=0$ if and only if $\beta\in \Omega^*(M;\mathfrak{su}(2))$, a straightforward computation yields the identities
\begin{equation*}
\Re(\omega\wedge\zeta)=(-1)^{k\ell}\Re(\zeta\wedge\omega)
\end{equation*}
and
\begin{equation} \label{eq: d Re}
d\Re(\omega\wedge\zeta)=\Re(d_\nabla\omega\wedge\zeta)+(-1)^{k}\Re(\omega\wedge d_\nabla\zeta)
\end{equation}
whenever $\omega\in \Omega^k(M;\mathfrak{su}(2))$ and $\zeta\in\Omega^l(M;\mathfrak{su}(2))$. As a consequence, for any $\phi\in\Omega^0(M;\mathfrak{su}(2)),$ we have
$$d\Re(\phi A)=\Re(d_\nabla\phi\wedge A)+\Re(\phi\,d_\nabla A)
=\Re((d\phi+[A,\phi])\wedge A)+\Re(\phi(F_\nabla+A\wedge A)).$$
Now we can compute that $\Re([A,\phi]\wedge A)=-2\Re(\phi A\wedge A)$, giving
\begin{equation} \label{equation: dRe}
    d\Re(\phi A)=\Re(d\phi\wedge A)+\Re(\phi F_\nabla)-\Re(\phi A\wedge A).
\end{equation}
We could have reached the same conclusion also by writing directly $dA=F_\nabla-A\wedge A$. We will denote by $\mathcal{A}(P)$ the space of smooth connections on a principal bundle $P$.  

\subsection{\for{toc}{The Yang--Mills--Higgs functionals and $Z(\Phi,A)$}\except{toc}{The Yang--Mills--Higgs functionals and $\bm{Z(\Phi,A)}$}}\hfill\\
Let us focus on the case of trivial bundles $P=SU(2)\times M$
and $E=\mathfrak{su}(2)\times M$
and, given $\nabla\in\mathcal{A}(P)$, let us identify it with the corresponding
one-form $A\in\Omega^1(M;\mathfrak{su}(2))$, writing $\nabla=d_A$ in computations.

Given a smooth section $\Phi\in \Gamma(E)$ and a smooth connection $A \in \mathcal{A}(P)$ on $P$, for any $\epsilon \in (0, 1)$, consider the $\SU(2)$-Yang--Mills--Higgs functionals
\begin{equation*}
    E_{\epsilon}(\Phi,A):=\int_M \left( \frac{1}{\epsilon}|d_A \Phi|^2+\epsilon|F_A|^2 \right) \, \dvol_g = \int_M e_{\epsilon}(\Phi, A) \, \dvol_g. 
\end{equation*}
Note that, cf.\ the previous section, $d_A \Phi = \nabla \Phi = d\Phi + [A, \Phi]$ in a local trivialization.
Although we do not make use of them in what follows, we observe that the Euler--Lagrange equations for $E_{\epsilon}$ are given by 
\begin{equation} \label{eq: EL}
    \begin{cases}
        d_A^\ast d_A \Phi = 0, \\
        \epsilon^2 d_A^\ast F_A + [\Phi, d_A \Phi] = 0, 
    \end{cases}
\end{equation}
where for $\alpha \in \Omega^k(M; E)$ we have 
\begin{equation*}
   (d_A \alpha)(X_0, X_1, \ldots, X_k) = \sum_{j = 1}^k (-1)^j (\nabla_{X_j} \alpha)(X_0, \ldots, \widehat{X}_j, \ldots, X_k), 
\end{equation*}
where $\widehat{X}_j$ means that $X_j$ has been omitted, and its adjoint
\begin{equation*}
    (d_A^\ast \alpha)(X_1, X_2, \ldots, X_{k - 1}) = - \sum_{j = 1}^n (\nabla_{e_j} \alpha)(e_j, X_1, \ldots, X_{k - 1}), 
\end{equation*}
for an orthonormal basis $\{e_j\}_{j = 1}^n$ of $T_pM$ at $p$. %Special solutions of \eqref{eq: EL} are the so-called \textit{monopoles}, i.e. solutions of the Bogomol'nyi equations 
%\begin{equation} \label{eq: monopole}
    %\ast d_A \Phi = \pm F_A.
%\end{equation} 
%Note that these first order equations imply the Euler--Lagrange equations \cref{eq: EL},
%parallel to connections with (anti-)self-dual curvature, i.e., Yang--Mills instantons.
%And indeed, both the second and first order equations can be seen as a dimensional reduction of the Yang--Mills situation: given a connection $d+\tilde A$ on $\R^3\times\R$, with
%$$\tilde A=A_1\,dx^1+A_2\,dx^2+A_3\,dx^3+\Phi\,dx^4$$
%invariant along $x_4$, we have
%$$F_{\tilde A}(\cdot,\de_4)=d_A\Phi,$$
%and moreover $\nabla\tilde=d_{\tilde A}=d+\tilde A$ is a Yang--Mills instanton if and only if $(\Phi,A)$ is a monopole, i.e.,
%$$\ast F_{\tilde A}=\pm F_{\tilde A}\quad\Longleftrightarrow\quad \ast d_A \Phi = \pm F_A.$$ 

Next, for any pair $A\in\mathcal{A}(P)$, $\Phi\in \Gamma(M;\mathfrak{su}(2))$, we introduce the three-form
\begin{equation} \label{eq: 3-form Z}
    Z(\Phi, A):= 2\Re (d_A \bar{\Phi}\wedge F_A).
\end{equation}
Throughout the paper, we identify $Z(\Phi, A)$ with an $(n - 3)$-current via the assignment
\begin{equation} \label{eq: def current Z}
    \langle Z(\Phi, A), \eta \rangle := \int_M Z(\Phi, A) \wedge \eta, 
\end{equation}
for all $\eta \in \Omega^{n - 3}(M)$. We observe next that $Z(\Phi,A)$ is bounded pointwise by the energy density $e_{\epsilon}(\Phi,A)$.

\begin{proposition}\label{trivial.ub}
    At each point we have
    $$\vert Z(\Phi, A) \vert \le 2|d_A\Phi||F_A| \leq \frac{1}{\epsilon}|d_A\Phi|^2+\epsilon|F_A|^2.$$
    %with equality if and only if $\ast d_A\Phi=\pm\epsilon F_A$.
\end{proposition}

\begin{remark}
    Given a 1-form $\alpha$ and a 2-form $\beta$ with values in $\mathfrak{su}(2)$,
    it is interesting to note that the inequality $|\alpha\wedge\beta|\le|\alpha||\beta|$
    is false in general. Rather, we will show that the weaker bound $|\Re(\alpha\wedge\beta)|\le|\alpha||\beta|$
    holds. Alternatively, this subtlety could be avoided by replacing the Euclidean norm on forms with the weaker comass norm, which would suffice to obtain the desired mass bounds in the $\epsilon\to 0$ limit.
\end{remark}

\begin{proof}
    Given $p\in M$, let $\alpha:=d_A\Phi(p)$ and $\beta:=F_A(p)$. We claim that
    \begin{equation}\label{claim.re}
    |\Re(\alpha\wedge\beta)|^2\le|\alpha|^2|\beta|^2.
    \end{equation}
    We fix an orthonormal basis $\{e_a\}_{a=1}^n$ at $p$
    and we denote
    $$\alpha_a:=\alpha[e_a],\quad\beta_{ab}:=\beta[e_a,e_b].$$
    Since $\alpha$ can be seen as a linear map with values into a 3-dimensional space,
    we can choose the basis in such a way that $\alpha_a=0$ for $a>3$.
    We compute that
    $$(\alpha\wedge\beta)[e_a,e_b,e_c]=\alpha_a\beta_{bc}-\alpha_b\beta_{ac}+\alpha_c\beta_{ab}.$$
    Hence, by Cauchy--Schwarz, we have
    $$|(\alpha\wedge\beta)[e_1,e_2,e_3]|^2\le|\alpha|^2\sum_{b<c\le3}|\beta_{bc}|^2.$$
    Moreover, when $b,c>3$, only the term $\alpha_a\beta_{bc}$ can be nonzero, and its squared norm equals the corresponding
    term $|\alpha_a|^2|\beta_{bc}|^2$ in the expansion of the right-hand side of \cref{claim.re}.
    
    Thus, it suffices to prove the bound
    $$\sum_{1\le a<b\le3} |\Re(\alpha_a\beta_{bc})-\Re(\alpha_b\beta_{ac})|^2\le\left(\sum_{a=1}^3|\alpha_a|^2\right)\left(\sum_{b=1}^3|\beta_{bc}|^2\right)$$
    for any fixed $c>3$.
    Setting $\gamma_a:=\beta_{ac}$, we are left to show the inequality
    $$\sum_{a<b}|\Re(\alpha_a\gamma_b)-\Re(\alpha_b\gamma_a)|^2\le\sum_{a,b}|\alpha_a|^2|\gamma_b|^2,$$
    where from now on we understand that indices vary only in $\{1,2,3\}$.
    
    Since the left-hand side can be rewritten as $\mz\sum_{a,b}|\Re(\alpha_a\gamma_b)-\Re(\alpha_b\gamma_a)|^2$,
    it is not affected by a change of the orthonormal basis $\{e_1,e_2,e_3\}$ (spanning the same three-dimensional subspace). By the spectral theorem,
    we can then assume that $\{\alpha_1,\alpha_2,\alpha_3\}$ is an orthogonal triple.
    Writing
    $$\alpha_a=\lambda_a \eta_a,$$
    with $\lambda_a\ge0$ and $\{\eta_1,\eta_2,\eta_3\}$ an orthonormal basis of $\Im(\mathbb{H})$, and decomposing each $\gamma_b$ as
    $$\gamma_b=:\sum_{a=1}^3\gamma_{a,b}\eta_a,$$
    so that $\gamma_{a,b}=\ang{\gamma_b,\eta_a}=\Re(\bar\eta_a\gamma_b)=-\Re(\eta_a\gamma_b)$, the claim becomes
    $$\sum_{a<b}|\lambda_a\gamma_{a,b}-\lambda_b\gamma_{b,a}|^2\le(\lambda_1^2+\lambda_2^2+\lambda_3^2)\sum_{a,b}|\gamma_{a,b}|^2.$$
    By Cauchy--Schwarz, the left-hand side is bounded by
    $$\sum_{a<b}(\lambda_a^2+\lambda_b^2)(|\gamma_{a,b}|^2+|\gamma_{b,a}|^2),$$
    which is clearly bounded by the right-hand side.
\end{proof}
In particular, the integral bound 
$$\|Z(\Phi,A)\|_{L^1(M)}\leq E_{\epsilon}(\Phi,A)$$
holds automatically; equivalently, under the identification \eqref{eq: def current Z}, we have 
$$\mathbb{M}(Z(\Phi, A)) = \|Z(\Phi,A)\|_{L^1(M)}\leq E_{\epsilon}(\Phi,A).$$

Next, we observe that $Z(\Phi,A)$ is an exact form: in particular, by combining \eqref{eq: d Re} with the Bianchi identity $d_AF_A=0,$ we see that 
\begin{equation}
    Z(\Phi,A)=2d(\beta(\Phi,A)),
\end{equation}
where
\begin{equation} \label{equation: definition of beta}
    \beta(\Phi,A):=\Re(\bar{\Phi}F_A).
\end{equation}
A central ingredient in the proof of Theorem \ref{thm: main theorem} is the simple observation that the integral of this two-form $\frac{1}{2\pi}\beta(\Phi,A)$ over $2$-cycles is approximately quantized where $|\Phi|\approx 1$. At the level of pointwise estimates, this can be proved by comparing $\beta(\Phi,A)$ with the pullback $\Phi^*(dA_{S^2})$ of the area element $dA_{S^2}$ in an arbitrary gauge as in \cite[Section II.5]{JaffeTaubes}, then applying the exterior derivative to arrive at the gauge-invariant identity
$$Z(\Phi,A)=\frac{1}{2}d\left(\Re(\bar{\Phi}\,d_A\Phi\wedge d_A\Phi)\right)\quad\text{where }|\Phi|\equiv 1.$$

For a more gauge-invariant %--though ultimately equivalent--
approach, one can instead relate $\beta(\Phi,A)$ to the curvature two-form of a $\U(1)$-bundle over the region $\{\Phi\neq 0\}$. Indeed, on the Euclidean 2-plane bundle $L\to \{\Phi\neq 0\}$ given by
$$L_p:=\{\Phi(p)\}^{\perp}\subset \mathfrak{su}(2),$$
note that the connection one-form $A\in \Omega^1(M;\mathfrak{su}(2))$ induces a metric-compatible connection $\nabla^{\perp}$ on $L$ via
$$\nabla^{\perp}\Psi:=d_A\Psi-\left\langle d_A\Psi,\frac{\Phi}{|\Phi|}\right\rangle \frac{\Phi}{|\Phi|}$$
for all $\Psi\in \Gamma(L)$. The curvature $F_{\nabla^{\perp}}\in \Omega^2(\{\Phi\neq 0\}, \mathfrak{so}(L))$ of $\nabla^{\perp}$ is then encoded, up to a sign depending on the choice of orientation, in a closed real-valued two-form on $\{\Phi\neq 0\}$, which we denote by $\omega(\Phi,A)$. With this notation in place, we record the following.

\begin{lemma}\label{quant.lem}
    On the set $U:=\{\Phi\neq 0\},$ the curvature two-form $\omega=\omega(\Phi,A)$ for the bundle $L\to U$ is given by
    \begin{equation}\label{omega.char}
        |\Phi|\omega(\Phi,A)=2\beta(\Phi,A)-\frac{1}{2}|\Phi|^{-2}\Re(\bar{\Phi}\,d_A\Phi\wedge d_A\Phi).
    \end{equation}
\end{lemma}
\begin{proof}
    On $U$, first define $\phi:=\frac{\Phi}{|\Phi|}$, and observe that the desired identity \eqref{omega.char} is equivalent to
    $$\omega(\Phi,A)=2\beta(\phi,A)-\frac{1}{2}\Re(\bar{\phi}\,d_A\phi\wedge d_A\phi).$$
    Next, since this is a local computation, we restrict our attention to a small ball $B\subset U$ around some arbitrary point in $U$, and make a change of gauge with respect to which $\phi\equiv \bfi$ on $B$. 

    In such a gauge, write $A=A^1\bfi+A^2\bfj+A^3\bfk$ for real-valued one-forms $A^1,A^2,A^3\in \Omega^1(B)$, and observe that
    $$d_A\phi=d_A\bfi=[A,\bfi]=-2A^2\bfk+2A^3\bfj$$
    and, by similar computations for $d_A\bfj$ and $d_A\bfk$, we see that
    $$\nabla^{\perp}\bfj=2A^1\bfk,\quad\nabla^{\perp}\bfk=-2A^1\bfj.$$
    In particular, fixing the orientation so that $\{\bfj, \bfk\}$ is positively oriented in this gauge, it follows that
    $$\omega(\Phi,A)=2dA^1=2d\Re(\bar{\phi}A).$$
    Applying \eqref{equation: dRe} and noting that $d\phi=0$ in this gauge, we then deduce that
    \begin{align*}
        d\Re(\bar{\phi}A)&=\beta(\phi,A)-\Re(\bar{\phi}A\wedge A)\\
        &=\beta(\phi,A)+\Re(\bfi A\wedge A)\\
        &=\beta(\phi,A)-2A^2\wedge A^3.
    \end{align*}
    On the other hand, we see that
    \begin{align*}
        \Re(\bar{\phi}\,d_A\phi\wedge d_A\phi)&=-\Re(\bfi(-2A^2\bfk+2A^3\bfj)\wedge (-2A^2\bfk+2A^3\bfj))\\
        &=8A^2\wedge A^3,
    \end{align*}
    so that 
    $$\omega(\Phi,A)=2d\Re(\bar{\phi} A)=2\beta(\phi,A)-\frac{1}{2}\Re(\bar{\phi}\,d_A\phi\wedge d_A\phi),$$
    as claimed.
\end{proof}

Now, if $S\subset U$ is any closed oriented surface in $U=\{\Phi\neq 0\}$, we observe that
$$\int_S\omega(\Phi,A)=2\pi e(\iota^*L),$$
where $\iota^*L\to S$ is the pullback bundle of $L$ under the inclusion $\iota: S\to U$, and $e(\iota^*L)$ is its Euler number. To see that this integral is in fact an integer multiple of $4\pi$, note that $L$ can be characterized as the pullback bundle
$$L=\phi^*(TS^2)$$
of the tangent bundle $TS^2$ under $\phi=\frac{\Phi}{|\Phi|}: U\to S^2\subset\mathfrak{su}(2)$
although, in the convention of Lemma \ref{quant.lem}, $TS^2$ is endowed with the \emph{opposite} of the canonical orientation. In particular, since $e(TS^2)=2$ for the canonical orientation, it follows from the naturality of Euler classes under pullback that $$e(\iota^*L)=-e([\phi \circ \iota]^*TS^2)=- 2\deg(\phi\circ \iota),$$
and therefore
\begin{equation}\label{beta.quant}
    \int_S\left[2|\Phi|^{-1}\beta(\Phi,A)-\frac{1}{2}|\Phi|^{-3}\Re(\bar{\Phi}\,d_A\Phi\wedge d_A\Phi)\right]=\int_S\omega(\Phi,A)=-4\pi \deg(\phi \circ \iota)
\end{equation}
for any $\Phi:M\to \mathfrak{su}(2)$ and any surface $S\subset\{\Phi\neq 0\}$.

\section{Liminf Inequality} 

Before beginning the proof of Theorem \ref{thm: main theorem}, we briefly recall some relevant terminology from geometric measure theory; see, e.g., \cite{SimonGMT} or \cite{Federer} for details.

First, we denote by $\mathcal{D}_k(M)$ the space of $k$-currents on $M$, i.e., the space of continuous linear functionals on $\Omega^k_c(M)$ with respect to the $C^{\infty}_c$ topology, and we follow the Euclidean convention that the \emph{mass} $\mathbb{M}(T)$ of a current $T\in \mathcal{D}_k(M)$ is given by
$$\mathbb{M}(T):=\sup\{\langle T,\omega\rangle \mid \omega\in \Omega^k_c(M)\text{ with }|\omega|\leq 1\},$$
while the \emph{boundary} $\partial T\in \mathcal{D}_{k-1}(M)$ is given by
$$\langle \partial T,\zeta\rangle:=\langle T,d\zeta\rangle.$$

As discussed in the previous section, the three-forms $Z(\Phi,A)$ can be identified with currents $[Z(\Phi,A)]\in\mathcal{D}_{n-3}(M)$ via integration, satisfying
$$\mathbb{M}([Z(\Phi,A)])\leq E_{\epsilon}(\Phi,A)$$
and
$$\langle \partial [Z(\Phi,A)], \zeta\rangle=\int_M Z(\Phi,A)\wedge d\zeta=\int_M d(Z(\Phi,A)\wedge \zeta)=0$$
for any $\zeta\in \Omega^2(M)$ supported in the interior of $M$. Moreover, for any $\zeta\in \Omega^{n-3}(M)$, it follows from \cref{trivial.ub} that
$$\langle [Z(\Phi,A)],\zeta\rangle \leq \int_M\left(\frac{1}{\epsilon}|d_A\Phi|^2+\epsilon|F_A|^2\right)|\zeta|.$$
By a standard application of the Banach--Alaoglu theorem (cf.\ \cite[Lemma 26.14]{SimonGMT}), these observations together immediately yield the following preliminary result.

\begin{lemma}\label{curr.lim}
Under the hypotheses of Theorem \ref{thm: main theorem}, there exists an $(n-3)$-current $T\in \mathcal{D}_{n-3}(M)$ such that, along a subsequence,
$$[Z(\Phi_j,A_j)]\rightharpoonup^* 4\pi T$$
as $(n-3)$-currents, and $T$ satifies
$$\spt(\partial T)\subset \partial M$$
and
$$4\pi |T|\leq \lim (\epsilon_j^{-1}|d_{A_j}\Phi_j|^2+\epsilon_j|F_{A_j}|^2)\,d\vol_g;$$
in particular, $4\pi \mathbb{M}(T)\leq \liminf_{j\to\infty}E_{\epsilon_j}(\Phi_j,A_j).$
\end{lemma}

Note that, away from $\de M$, the current $T$ in Lemma \ref{curr.lim} is automatically \emph{normal} in the sense of \cite{Federer}: that is, $T$ and $\partial T$ both have finite mass in every compact subset of the interior of $M$. To prove Theorem \ref{thm: main theorem}, what remains is to show that the limit current $T$ in Lemma \ref{curr.lim} is an \emph{integral} $(n-3)$-current in $M\setminus\de M$, that is, 
$$T(\omega)=\int_{\Sigma}\theta(x)\langle \omega(x),\xi(x)\rangle \, d\mathcal{H}^{n-3}(x)$$
for some $(n-3)$-rectifiable set $\Sigma\subset M$, with $\theta: \Sigma\to \mathbb{Z}$ an $\mathcal{H}^{n-3}$-measurable weight and $\xi(x)$ a volume form for $T_x\Sigma$ at $\mathcal{H}^{n-3}$-a.e.\ $x\in \Sigma$. 

Further, we observe that we can assume
\begin{equation}\label{mod.at.most.1}
    |\Phi_j|\le1
\end{equation}
in the sequel: indeed, letting $\tilde\Phi_j:=\Phi_j$ on $\{|\Phi_j|\le1\}$ and
$\tilde\Phi_j:=\frac{\Phi_j}{|\Phi_j|}$ on $\{|\Phi_j|>1\}$, we have
$$\|\tilde\Phi_j-\Phi_j\|_{L^2(M)}\le\||\Phi_j|-1\|_{L^2(M)}=o(\epsilon_j^{1/2}),$$
giving in particular
$$\|\beta(\tilde\Phi_j,A_j)-\beta(\Phi_j,A_j)\|_{L^1(M)}\le o(\epsilon_j^{1/2})\|F_{A_j}\|_{L^2(M)}=o(1),$$
and thus
\begin{equation}\label{same.lim}
    \lim_{j\to\infty}Z(\tilde\Phi_j,A_j)=\lim_{j\to\infty}Z(\Phi_j,A_j).
\end{equation}
Since $\langle\frac{\Phi_j}{|\Phi_j|},d_{A_j}\frac{\Phi_j}{|\Phi_j|}\rangle=0$ we also have
$$E_{\epsilon_j}(\tilde\Phi_j,A_j)\le E_{\epsilon_j}(\Phi_j,A_j),$$
and hence we can guarantee \eqref{mod.at.most.1} by replacing $\Phi_j$ with $\tilde\Phi_j$
(which can be slightly perturbed to a smooth function, although Lipschitz regularity
is certainly enough for the computations in this section).

The remainder of this section is devoted to the proof of this integrality result, completing the proof of Theorem \ref{thm: main theorem}. We start by reducing the general situtation to the three-dimensional case by slicing, and then we prove it in three dimensions.

\subsection{Reduction to dimension 3}\hfill\\
%Let $4\pi T$ be any subsequential limit of $T_j:=Z(\Phi_j,A_j)$, which we view dually as an $(n-3)$-current. By \cref{trivial.ub}, we have $4\pi|T|\le\mu$.
%Thus, it suffices to show that $T$ is an integral current.
Since integrality of the current $T$ in Lemma \ref{curr.lim} is a local property,
by pulling back sections and connections via arbitrary local parametrizations,
we can replace $M$ with the cube $(-2,2)^n$. Writing any point $x\in(-2,2)^n$
as $x=(y,z)$, with $y\in(-2,2)^{n-3}$ and $z\in(-2,2)^3$,
let
$$\pi:(-2,2)^n\to(-2,2)^{n-3},\quad \pi(y,z):=y.$$
%Moreover, let $\chi\in C^1_c((-2,2)^n)$ be an arbitrary function
%and consider the modified currents $\chi T_j$. Note that these are normal currents,
%with the uniform bound
%$$\mathbb{N}(\chi T_j)=\mathbb{M}(\chi T_j)+\mathbb{M}(\de(\chi T_j))\le C(n)\|\chi\|_{C^1}\mathbb{M}(T_j).$$
%Thus, by approximation, we can take $0<\chi\le1$ on $(-2,2)^n$,
%vanishing at the boundary.
Given $y\in (-2,2)^{n-3}$ and a normal $(n-3)$-current $S$ in $(-2,2)^n$, we let
$S^y=\langle S,\pi,y\rangle$ denote its slice relative to the map $\pi$ and the value $y$ (in the sense of \cite[Section 4.3]{Federer}), so that $S^y$ is defined for a.e.\ $y$ and is a normal 0-current in $(-2,2)^3$. Letting $T_j=\frac{1}{4\pi}[Z(\Phi_j,A_j)]$, in the setting of Lemma \ref{curr.lim}, we observe next that the operation of taking slices commutes with the weak convergence $T_j\rightharpoonup^*T$.
\begin{lemma}\label{slice.conv}
Under the assumptions of Lemma \ref{curr.lim} with $M=(-2,2)^n$, after passing to a further subsequence, we have the weak convergence
$$T_j^y\rightharpoonup^* T^y$$
for almost every $y\in (-2,2)^{n-3}.$
\end{lemma}
\begin{proof}
Let $\chi\in C^1_c(M)$, where now $M=(-2,2)^n$, and let
$\tilde T_j:=\chi T_j$, which we can view as a normal current in $\R^n$.
Since $\tilde T_j$ and $\de\tilde T_j$ have equibounded mass, classical compactness results from geometric measure theory (see \cite[Theorem 4.2.17]{Federer}) imply that the weak convergence $\tilde T_j\rightharpoonup^*\tilde T:=\chi T$ can be upgraded to convergence in flat norm
$$\lim_{j\to\infty}\mathbb{F}(\tilde T_j-\tilde T)=0,$$
where we recall that
\begin{align*}
    \mathbb{F}(S)&:=\sup\{S(\omega)\mid \omega \in \Omega^m_c(\R^n)\text{ with }\max\{\|\omega\|_{C^0},\|d\omega\|_{C^0}\}\le1\}\\
    &\phantom{:}=\inf\{\mathbb{M}(S-\partial C)+\mathbb{M}(C)\mid C\in \mathcal{D}_{n-2}(\R^n)\}.
\end{align*}
Now, by standard estimates for the slicing operation (see \cite[Section 4.3.1]{Federer}) we have
$$\int_{y\in (-2,2)^{n-3}}\mathbb{F}(S^y)\leq \mathbb{F}(S),$$
so that
$$\lim_{j\to\infty}\int_{y\in (-2,2)^{n-3}}\mathbb{F}(\tilde T_j^y-\tilde T^y)=0.$$
Thus, up to a further subsequence, for almost every $y\in (-2,2)^{n-3}$ we have
$$\tilde T_j\rightharpoonup^*\tilde T.$$
Since $\chi$ was arbitrary, a standard argument implies the existence of a further subsequence such that
$$T_j^y\rightharpoonup^* T^y$$
for almost every $y\in (-2,2)^{n-3}$.
\end{proof}

To reduce the integrality question to the three-dimensional case, we make essential use of the following result, due in various forms to White, Jerrard, and Ambrosio--Kirchheim. 

\begin{proposition}\label{slice.red}\cite{White, Jerrard.slices, AK}
    A normal $k$-current $S$ in $\mathbb{R}^n$ is integral if and only if, for every projection $P$ onto a $k$-dimensional coordinate subspace of $\mathbb{R}^n$, the slices $\langle S,P,y\rangle$ with respect to the projection $\pi: \mathbb{R}^n\to P$ are integral $0$-currents for almost every $y\in \mathbb{R}^k$.
\end{proposition}

 To complete the proof that $T$ is an integral $(n-3)$-current in the situation of Lemma \ref{curr.lim} with $M=(-2,2)^n$, it therefore suffices to show that the slices $T^y$ are integral $0$-currents in $(-2,2)^3$ for almost every $y$. On the other hand, letting $\Phi_j^y(z):=\Phi_j(y,z)$
and $A_j^y(z):=(\iota^y)^*A_j$, where $\iota^y:(-2,2)^3\to(-2,2)^n$ is given by $\iota^y(z):=(y,z)$,
by Fatou's lemma we have
$$ \int_{(-2,2)^{n-3}}\left(\liminf_{\epsilon\to0}\int_{\{y\}\times(-2,2)^3} e_\epsilon(\Phi_j,A_j)\right)\,dy\le C.$$
Thus, for a.e.\ $y\in(-2,2)^{n-3}$, we can find a subsequence (depending on $y$, but not relabeled) such that
$$ \int_{(-2,2)^3} e_\epsilon(\Phi_j^y,A_j^y)\le C(y)<\infty. $$
Moreover, we clearly have $T_j^y=[Z(\Phi_j^y,A_j^y)]$,
where we identify $Z(\Phi_j^y,A_j^y)$ with its dual 0-current in $\mathbb{R}^3$ and, by Lemma \ref{slice.conv}, we know moreover that
$$T^y=\lim_{j\to\infty}T_j^y$$
for almost every $y$ as well. Hence, it is enough to show that $\frac{1}{4\pi}Z(\Phi_j^y,A_j^y)$ converges to an integral 0-current in $(-2,2)^3$ (i.e., a finite linear combination of Dirac masses, with integer coefficients). In the next section, we demonstrate the desired integrality of the limit current $T$ in the setting of Lemma \ref{curr.lim} in dimension three, completing the proof of Theorem \ref{thm: main theorem}.

\subsection{Liminf inequality: proof in dimension three}\hfill\\
Assume now that $[-1,1]^3\subset M\subset\R^3$;
in the sequel, with a slight abuse of notation, we will write $\Phi_\epsilon$, $A_\epsilon$, and $\epsilon$
in place of $\Phi_j$, $A_j$, and $\epsilon_j$.

It is enough to show that $Z(\Phi_\epsilon,A_\epsilon)$ converges
to ($4\pi$ times) an integral cycle $T$ along a subsequence, on the smaller open set $(-1,1)^3$.

We first claim that a good change of gauge exists on sufficiently small cubes. The following is a consequence of the classical result of Uhlenbeck \cite[Theorem 2.1]{UhlenbeckLP}. 

\begin{proposition}
    There exists a small constant $c_0>0$ with the following property:
    given a smooth connection one-form $A$ on a closed cube $Q\subset M$ (of any sidelength), if $\int_Q|F_A|^{3/2}\le c_0$ then we can find $\sigma\in W^{1,6}(Q,\Sp(1))$ such that
    the gauge-changed one-form
    $$A^\sigma=-(d\sigma)\sigma^{-1}+\sigma A\sigma^{-1}$$
    satisfies
    $$\|A^\sigma\|_{L^3(Q)}+\|DA^\sigma\|_{L^{3/2}(Q)}\le C\|F_A\|_{L^{3/2}(Q)},$$
    as well as
    $$\|A^\sigma\|_{L^6(Q)}+\|DA^\sigma\|_{L^{2}(Q)}\le C\|F_A\|_{L^{2}(Q)}.$$
\end{proposition}

In our setting, it suffices to take cubes of sidelength $c_1\epsilon$, with $c_1$ small enough depending on the total energy: indeed, by H\"older's inequality we then have
$$ \int_Q|F_{A_\epsilon}|^{3/2}\le\left(\int_Q|F_{A_\epsilon}|^2\right)^{3/4}|Q|^{1/4}\le \epsilon^{-3/4}E_\epsilon(\Phi_\epsilon,A_\epsilon)^{3/4}(c_1\epsilon)^{3/4}\le c_0. $$
%We can extend $\sigma$ in a standard way to a $W^{1,6}$ function defined on all of $\mathbb{T}^3$.

\begin{lemma}
    Dropping the subscript $\epsilon$,
    there exists a section $\tilde\Phi$ on $M$ with
    $$[|\tilde\Phi|]_{C^{0,1/2}(Q)}\le C\epsilon^{-1/2}E_\epsilon(\Phi,A)^{1/2}$$
    on each cube $Q$ of sidelength $\epsilon$ and $\operatorname{dist}(Q,\R^3\setminus M)\ge\epsilon$,
    as well as
    $$E_\epsilon(\tilde\Phi,A)\le CE_\epsilon(\Phi,A)$$
    and
    $$\|\tilde\Phi-\Phi\|_{L^2(M)}^2\le C\epsilon^3E_\epsilon(\Phi,A).$$
\end{lemma}

\begin{proof}
    In this proof, we let $N:=\lceil\frac{2}{c_1\epsilon}\rceil$
    and $\ell:=\frac{1}{N}$, so that $2\ell\le c_1\epsilon$, as well as $\delta:=\ell/10$,
    which is comparable with $\epsilon$.
    We first regularize $\Phi$ on the slab $S:=[-1,1]\times[-\ell,\ell]^2$
    (actually, on the slightly thinner one $\tilde S:=[-1+2\delta,1-2\delta]\times[-\ell+2\delta,\ell-2\delta]^2$).

    We cover $S$ with $2N-1$ cubes $Q_1,\dots,Q_{2N-1}$ of sidelength $2\ell$, given by
    $$Q_j:=[-1+j\ell-\ell,-1+j\ell+\ell]\times[-\ell,\ell]\times[-\ell,\ell].$$
    We apply the previous lemma for each $j$, finding a gauge transformation
    $\sigma_j:Q_j\to\Sp(1)$ such that
    $$\|A^{\sigma_j}\|_{L^6(Q_j)}\le C\|F_A\|_{L^2(Q_j)}.$$
    Let $\lambda_j:M\to[0,\delta]$ be given by
    $$\lambda_j(x):=\min\{\delta,\operatorname{dist}(x,M\setminus Q_j)\},$$
    so that $\lambda_j$ is $1$-Lipschitz and
    $$\lambda_j=\delta\quad\text{on }\tilde Q_j,$$
    where $\tilde Q_j\subset Q_j$ consists of all points with distance at least $\delta$ from $\partial Q_j$.
    We also fix a smooth function $\chi:\R^3\to[0,\infty)$ supported in the ball $B_{1/2}(0)\subset\R^3$, with $\int_{\R^3}\chi=1$, so that
    $$|d\lambda_j(x)||y|\le \frac12$$
    for a.e.\ $x\in M$ and $y\in\operatorname{spt}(\chi)$.
    Note that this guarantees that the map $x\mapsto x+\lambda_j(x)y$ is bi-Lipschitz and equals the identity on $M\setminus Q_j$, for all $y\in\operatorname{spt}(\chi)$.
    
    We now let $\Phi'$ be the new section obtained by considering all cubes $Q_{j}$ with $j$ odd
    and replacing $\Phi(x)$ with
    $$\Phi'(x):=\sigma_j(x)^{-1}\int_{B_{1/2}}(\sigma_j\Phi)(x+\lambda_j(x)y)\chi(y)\,dy$$
    on $Q_j$, leaving $\Phi$ unchanged on $M\setminus S$
    (note that $S$ is the disjoint union of such cubes, neglecting their boundaries,
    where clearly $\Phi'=\Phi$).

Using \eqref{mod.at.most.1}, on $Q_j$ we deduce the bound
$$\|d(\sigma_j\Phi)\|_{L^2(Q_j)}\le\|d_A\Phi\|_{L^2(Q_j)}+\|A^{\sigma_j}\|_{L^2(Q_j)}
\le\|d_A\Phi\|_{L^2(Q_j)}+C\epsilon\|F_A\|_{L^2(Q_j)},$$
since we have $\|A^{\sigma_j}\|_{L^2(Q_j)}\le|Q_j|^{1/3}\|A^{\sigma_j}\|_{L^6(Q_j)}$ by H\"older.
Moreover, thanks to the fact that $x\mapsto x+\lambda_j(x)y$ is bi-Lipschitz and the convexity of the $L^2$-norm, we have
$$\|d(\sigma_j\Phi')\|_{L^2(Q_j)}\le C\|d(\sigma_j\Phi)\|_{L^2(Q_j)},$$
giving the gauge-invariant bound
$$\|d_A\Phi'\|_{L^2(Q_j)}\le C\|d_A\Phi\|_{L^2(Q_j)}+C\epsilon\|F_A\|_{L^2(Q_j)},$$
and thus
\begin{equation}\label{en.same}
    E_\epsilon(\Phi',A;S)\le CE_\epsilon(\Phi,A;S).
\end{equation}

As for the distance in $L^2$, it suffices to bound
$$x\mapsto(\sigma_j\Phi)(x+\lambda_j(x)y)-(\sigma_j\Phi)(x)$$
in $L^2$ for each $y\in\operatorname{spt}(\chi)$.
To do this, we write
$$|(\sigma_j\Phi)(x+\lambda_j(x)y)-(\sigma_j\Phi)(x)|\le\int_0^1 |d(\sigma_j\Phi)(x+s\lambda_j(x)y)||\lambda_j(x)y|\,dy.$$
Since $|\lambda_j(x)y|\le\epsilon$, using again the fact that $x\mapsto x+s\lambda_j(x)y$ is bi-Lipschitz we get
$$\int_{Q_j}|(\sigma_j\Phi)(x+\lambda_j(x)y)-(\sigma_j\Phi)(x)|^2\,dx
\le C\epsilon^2\int_{Q_j}|d(\sigma_j\Phi)|^2\le C\epsilon^3E_\epsilon(\Phi,A;Q_j).$$
In particular, summing over $j$ we get the bound
\begin{equation}\label{l2.shift}
    \|\Phi'-\Phi\|_{L^2(M)}^2\le C\epsilon^3E_\epsilon(\Phi,A;S).
\end{equation}

Since $\lambda_j=\delta$ on $\tilde Q_j$, a simple change of variables gives
$$|d(\sigma_j\Phi')|(x)\le C\epsilon^{-3}\int_{Q_j}|d(\sigma_j\Phi)|
\le C\epsilon^{-3/2}\|d(\sigma_j\Phi)\|_{L^2(Q_j)}
\le C\epsilon^{-1}E_\epsilon(\Phi,A)^{1/2}$$
for $x\in \tilde Q_j$ (and $j$ odd). In particular, in terms of the $C^{0,1/2}$-seminorm we have
$$[\sigma_j\Phi']_{C^{0,1/2}(\tilde Q_j)}\le C\epsilon^{-1/2}E_\epsilon(\Phi,A)^{1/2}$$
for $j$ odd, and thus the same bound holds for $|\Phi'|$, which is gauge-invariant.

We now repeat the procedure on the cubes $Q_j$ with $j$ even,
obtaining a new section $\Phi''$. This section satisfies the same bounds \cref{en.same} and \cref{l2.shift}, as well as
$$[\sigma_j\Phi'']_{C^{0,1/2}(\tilde Q_j)}\le C\epsilon^{-1/2}E_\epsilon(\Phi,A)^{1/2}$$
for $j=2,4,\dots,2N-2$ even, and thus
$$[|\Phi''|]_{C^{0,1/2}(\tilde Q_j)}\le C\epsilon^{-1/2}E_\epsilon(\Phi,A)^{1/2}$$
for the same indices.

Letting $\tilde S:=[-1+2\delta,1-2\delta]\times[-\ell+2\delta,\ell-2\delta]^2$,
we claim that, in fact, the same holds also on $\tilde Q_j\cap\tilde S$ for $j$ odd; we focus on the case $j\neq1,2N-1$,
since the endpoint case is analogous.

Let us fix $j\in\{3,\dots,2N-3\}$; it is enough to show H\"older continuity on
$$(\tilde Q_j\cap\tilde S)\setminus(\tilde Q_{j-1}\cup\tilde Q_{j+1})
\subseteq[-1+j\ell-\delta,-1+j\ell+\delta]\times[-\ell+2\delta,\ell-2\delta]^2.$$
We deal only with $R:=[-1+j\ell-\delta,-1+j\ell]\times[-\ell+2\delta,\ell-2\delta]^2$, since the argument for the second half is identical.
We let $\Psi:=\sigma_j\Phi'$, defined and H\"older continuous on $\tilde Q_j$, and
$$\sigma:Q_{j-1}\cap Q_j\to \Sp(1),\quad \sigma:=\sigma_{j-1}\sigma_j^{-1},$$
the change-of-gauge function. Note carefully that,
by definition of $\lambda_{j-1}$,
we have $\lambda_{j-1}\le\delta$ and $x+\lambda_{j-1}(x)y\in Q_{j-1}$ for $x\in Q_{j-1}$ and $y\in\operatorname{spt}(\chi)$, giving
$$x+\lambda_{j-1}(x)y\in [j\ell-2\delta,j\ell]\times[-\ell+\delta,\ell-\delta]^2
\subseteq Q_{j-1}\cap\tilde Q_j\quad\text{for all }x\in R.$$
By definition, for $x\in R$ we have
$$(\sigma_j\Phi'')(x)=\sigma(x)^{-1}(\sigma_{j-1}\Phi'')(x)=\sigma(x)^{-1}\int_{B_1}(\sigma_{j-1}\Phi')(x+\lambda_{j-1}(x)y)\chi(y)\,dy,$$
and hence at these points we can write
$$(\sigma_j\Phi'')(x)=\sigma(x)^{-1}\int_{B_1}(\sigma\Psi)(x+\lambda_{j-1}(x)y)\chi(y)\,dy.$$

Further, since on $Q_{j-1}\cap Q_j$ we have the identity
$$A^{\sigma_{j-1}}=-(d\sigma)\sigma^{-1}+\sigma A^{\sigma_j}\sigma^{-1},$$
we can bound
$$\|d\sigma\|_{L^6(Q_{j-1}\cap Q_j)}\le\|A^{\sigma_{j-1}}\|_{L^6(Q_{j-1})}+\|A^{\sigma_{j}}\|_{L^6(Q_{j})}
\le C\|F_A\|_{L^2(Q_{j-1}\cup Q_j)}.$$
Thus, by Morrey's inequality, we obtain
$$[\sigma]_{C^{0,1/2}(Q_{j-1}\cap Q_j)}\le C\|F_A\|_{L^2(Q_{j-1}\cup Q_j)}\le C\epsilon^{-1/2}E_\epsilon(\Phi,A)^{1/2}.$$
Since $\Psi$ satisfied the same bound on $\tilde Q_j$, the same holds for the product $\sigma\Psi$
on $Q_{j-1}\cap\tilde Q_j$. Since $x\mapsto x+\lambda_{j-1}(x)y$ is Lipschitz, the claim follows.

As a consequence, $\Phi''$ is H\"older continuous on the smaller slab $\tilde S$.
Next, for $j=1,\dots,2N-1$, we consider the translated slabs
$$S_j:=[-1,1]\times [-1+j\ell-\ell,-1+j\ell+\ell]\times[-\ell,\ell],$$
whose union is $[-1,1]\times[-1,1]\times[-\ell,\ell]$,
and we perform the previous replacement simultaneously for all odd indices $j$,
obtaining a section which is H\"older continuous on each
$$\tilde S_j:=[-1+2\delta,1-2\delta]\times [-1+j\ell-\ell+2\delta,-1+j\ell+\ell-2\delta]\times[-\ell+2\delta,\ell-2\delta],$$
and then for all the even indices, ending up with a section which is H\"older continuous
on
$$[-1+3\delta,1-3\delta]\times[-1+3\delta,1-3\delta]\times[-\ell+3\delta,\ell-3\delta],$$
satisfying the same integral bounds.
Finally, we do the same for the third coordinate, obtaining the desired section $\tilde\Phi$,
which is H\"older continuous on $[-1+4\delta,1-4\delta]^3$.
\end{proof}

As a consequence, we get
$$ \|\beta(\tilde\Phi_{\epsilon},A_{\epsilon})-\beta(\Phi_{\epsilon},A_{\epsilon})\|_{L^1}
\le \|\tilde\Phi_{\epsilon}-\Phi_{\epsilon}\|_{L^2}\|F_{A_{\epsilon}}\|_{L^2}
\le C\epsilon E_\epsilon(\Phi_{\epsilon},A_{\epsilon}), $$
so that $Z(\Phi_{\epsilon},A_{\epsilon})$ has the same limit as $Z(\tilde\Phi_{\epsilon},A_{\epsilon})$, and we can replace $\Phi$ with $\tilde\Phi$ in the sequel.

Now fix $\ell>0$ small, independent of $\epsilon$.
By a simple averaging argument, we can
select a grid of fixed sidelength $\ell$ in $(-1+\epsilon,1-\epsilon)^3$,
such that the 3-skeleton includes $(-1+2\ell,1-2\ell)^3$
and on the 2-skeleton $S_2$ we have
\begin{equation}\label{2skel.est}\int_{S_2}e_\epsilon(\Phi_{\epsilon},A_{\epsilon})\le\frac{C}{\ell},\quad\int_{S_2}(1-|\Phi_{\epsilon}|)^2=o(\epsilon/\ell).
\end{equation}
On each plane $\Sigma\subset S_2$ we then have
$$\int_\Sigma|1-|\Phi_{\epsilon}||^2+|d|\Phi_{\epsilon}||^2\le \frac{C\epsilon}{\ell}$$
for $\epsilon$ small enough, which for each $p\in(1,\infty)$ gives
$$\|1-|\Phi_{\epsilon}|\|_{L^p(\Sigma)}^2\le C(p)\frac{\epsilon}{\ell}$$
by Sobolev embedding.
If we had $|\Phi_{\epsilon}|<\frac34$ at some $x\in\Sigma$ then, recalling that $|\Phi_{\epsilon}|$ is H\"older continuous at scale $\epsilon$, we would have $|\Phi_{\epsilon}|<\frac78$ on a ball $B_{c\epsilon}(x)\cap\Sigma$,
where $c$ depends also on (an upper bound on) the energy; hence,
for any fixed $p\in(4,\infty)$ we would get
$$\|1-|\Phi_{\epsilon}|\|_{L^p(\Sigma)}^2\ge c(p)\epsilon^{4/p}.$$
Hence, we would reach the inequality
$$\epsilon^{4/p}\le C(p)\frac{\epsilon}{\ell},$$
which is impossible for $\epsilon$ small enough (as $\ell>0$ is fixed). 

Thus, $|\Phi_{\epsilon}|\ge\frac34$ on the 2-skeleton $S_2$. %After the usual modification [DESERVES BRIEF COMMENT], we can rather assume that $|\Phi|=1$ on $S_2$.
%Proceeding as in \cite{ABO}, we can project the 0-current $[Z(\Phi_{\epsilon},A_{\epsilon})]$ to the dual 0-skeleton (the set of centers of all 3-faces), obtaining a current $T_{\epsilon,\ell}$ with
%\begin{equation}\label{f.close}
%\mathbb{F}([Z(\Phi_{\epsilon},A_{\epsilon})]-T_{\epsilon,\ell})\le C\ell.
%\end{equation}
%
Calling $x_Q$ the center of each $3$-face $Q$, we now define the $0$-current
$$T_{\epsilon,\ell}:=\sum_Q\left(\int_Q Z(\Phi_{\epsilon},A_{\epsilon})\right)\delta_{x_Q}.$$
The multiplicity of $T_{\epsilon,\ell}$ at each $x_Q$ equals
$$\int_Q Z(\Phi_{\epsilon},A_{\epsilon})=\int_{\de Q}2\beta(\Phi_{\epsilon},A_{\epsilon});$$
since $|\Phi_{\epsilon}|\ge \frac34$ on $\de Q$, we observe that
$$|\beta(\Phi_{\epsilon},A_{\epsilon})-|\Phi_{\epsilon}|^{-1}\beta(\Phi_{\epsilon},A_{\epsilon})|\leq C (1-|\Phi_{\epsilon}|)|F_{A_{\epsilon}}|,$$
which together with \eqref{beta.quant} yields
$$\left|\int_Q Z(\Phi_{\epsilon},A_{\epsilon})+4\pi \deg(\phi_{\epsilon}|_{\de Q})\right|\le C\int_{\de Q}(1-|\Phi_{\epsilon}|)|F_{A_{\epsilon}}|+|d_{A_{\epsilon}}\Phi_{\epsilon}|^2.$$
In particular, denoting by $S_{\epsilon,\ell}$ the integral current
$$S_{\epsilon,\ell}:=\sum_Q -\deg(\phi_{\epsilon}|_{\de Q})\delta_{x_Q},$$
it follows that
\begin{align*}
    \mathbb{M}(T_{\epsilon,\ell}-4\pi S_{\epsilon,\ell})&\leq C\int_{S_2}(1-|\Phi_{\epsilon}|)|F_{A_{\epsilon}}|+|d_{A_{\epsilon}}\Phi_{\epsilon}|^2\\
    &\leq C\left(\int_{S_2}\frac{(1-|\Phi_{\epsilon}|)^2}{\epsilon}\right)^{1/2}\left(\int_{S_2}e_{\epsilon}(\Phi_{\epsilon},A_{\epsilon})\right)^{1/2}+C\epsilon \int_{S_2}e_{\epsilon}(\Phi_{\epsilon},A_{\epsilon}),
\end{align*}
by an application of the Cauchy--Schwarz inequality. By \eqref{2skel.est}, we know moreover that the right-hand side of the preceding estimate vanishes as $\epsilon\to 0$, so that
\begin{equation}\label{int.lim}\lim_{\epsilon\to 0}\mathbb{M}(T_{\epsilon,\ell}-4\pi S_{\epsilon,\ell})=0.
\end{equation}
Moreover, since 
$$\mathbb{M}(T_{\epsilon,\ell})=\sum_Q\Big|\int_Q Z(\Phi_{\epsilon},A_{\epsilon})\Big|\le E_{\epsilon}(\Phi_{\epsilon},A_{\epsilon})\leq C$$
as $\epsilon\to 0$, after passing to a further subsequence, we see that $\frac{1}{4\pi}T_{\epsilon,\ell}$ converges as $\epsilon\to 0$ to a limit $S_{\ell}$, which, by \eqref{int.lim}, must be integral. Moreover, since the mass $\mathbb{M}(S_{\ell})\leq C$ is bounded independent of $\ell$, we can pass to another subsequential limit $S_{\ell}\to S$ as $\ell \to 0$, with $S$ an integral current. 

%Finally, letting $T:=\frac{1}{4\pi}\lim_{\epsilon\to 0}[Z(\Phi_{\epsilon},A_{\epsilon})]$,
%it follows from \eqref{f.close} and the discussion above that 
Finally, letting $U_\ell:=(-1+2\ell,1-2\ell)^3$, for any $\zeta\in C^1_c(U_\ell)$ we have
\begin{align*}
    \int_{U_\ell} Z(\Phi_\epsilon,A_\epsilon)\wedge\zeta&=\sum_Q\left(\int_Q Z(\Phi_{\epsilon},A_{\epsilon})\right)\zeta(x_Q)+O(\ell\|d\zeta\|_{C^0(U_\ell)})\\
    &=\langle T_{\epsilon,\ell},\zeta\rangle+O(\ell\|d\zeta\|_{C^0(U_\ell)}).
\end{align*}
$$ $$
Letting $T:=\frac{1}{4\pi}\lim_{\epsilon\to 0}[Z(\Phi_{\epsilon},A_{\epsilon})]$,
we deduce that
$$\|T-S_{\ell}\|_{C^1_c(U_\ell)^*}\leq C\ell$$
for all $\ell>0$, and therefore $T=S$ is indeed an integral $0$-current, as desired.

%The first term gives an integer (times $4\pi$), while the term containing $d\rho$ vanishes as $\rho\equiv 1$ on $\de Q\subset S_2$. Moreover, the sum of the other terms (in absolute value) as $Q$ varies is bounded by
%$$\int_{S_2}2(1-\rho)|F_A|+O(|d_A\Phi|^2),$$
%which is infinitesimal by our choice of the grid. Also,
%Hence, as %$\epsilon\to0$ (up to subsequences), on the open set $(-1+2\ell,1-2\ell)^3$ the 0-current $T_{\epsilon,\ell}$ converges to a linear combination $\sum_Q \alpha_Q\delta_{x_Q}$ with $\alpha_Q\in 4\pi\Z$. Letting $\ell\to0$, we obtain the claim.

\section{Limsup inequality} \label{sec.limsup}
The aim of this section is to prove Theorem \ref{thm: recovery}, the `$\limsup$' part of our $\Gamma$-convergence result. We begin with a useful computational lemma.

\begin{proposition} \label{proposition: dAPhi = 0}
    Given an open set $U\subseteq M^n$ and
    a smooth unit section $\Phi:U\to\mathfrak{su}(2)\subset\mathbb{H}$, for an arbitrary $\alpha\in\Omega^1(U,\R)$ the connection one-form
    $$ A:=\mz\Phi^{-1}\,d\Phi+\Phi\alpha=-\mz\Phi\,d\Phi+\Phi\alpha $$
    satisfies
    $$ d_A\Phi=0,\quad F_A=-\frac14d\Phi\wedge d\Phi+\Phi\,d\alpha. $$
\end{proposition}

\begin{proof}
    Recall that, since $\bar\Phi=-\Phi$, we have $\Phi^2=-\Phi\bar\Phi=-1$. As a consequence, $\Phi^{-1}=-\Phi$ and
    \begin{equation}\label{idunit}
        \Phi\,d\Phi=-d\Phi\,\Phi.
    \end{equation}
    Thus,
    $$ d_A\Phi=d\Phi+[A,\Phi]=d\Phi+A\Phi-\Phi A=\mz d\Phi-\mz\Phi\,d\Phi\,\Phi=\mz(1+\Phi^2)\,d\Phi=0 $$
    and, since $\Phi\alpha\wedge\Phi\alpha=\Phi^2\alpha\wedge\alpha=0$, we also have
    $$ F_A=dA+A\wedge A=-\mz d\Phi\wedge d\Phi+d(\Phi\alpha)+\frac14\Phi\,d\Phi\wedge\Phi\,d\Phi-\mz(\Phi\,d\Phi\wedge\Phi\alpha+\Phi\alpha\wedge\Phi\,d\Phi). $$
    Using \cref{idunit}, we obtain
    $$-\mz(\Phi\,d\Phi\wedge\Phi\alpha+\Phi\alpha\wedge\Phi\,d\Phi)=-d\Phi\wedge\alpha,$$
    from which the conclusion follows.
\end{proof}

%Now note that
%$$\Phi^{-1}\,d\Phi\wedge d\Phi=-\Phi\,d\Phi\wedge d\Phi$$
%is in fact real-valued, since its conjugate equals
%$$d\bar\Phi\wedge d\bar\Phi\,\bar\Phi=-d\Phi\wedge d\Phi\,\Phi=-\Phi\,d\Phi\wedge d\Phi$$
%(the last equality follows from a couple of applications of \cref{idunit}).
%Hence,
%$$\Phi^{-1}\,d\Phi\wedge d\Phi=Re(\bar\Phi\,d\Phi\wedge d\Phi)=2\Phi^*(\omega)$$
%where $\omega$ is the area form on $S^2$. In particular, since $\omega$ is closed, we also have
%$$d\Phi^*(\omega)=0.$$

%\begin{proposition}
    %If moreover $\Phi^*(\omega)$ is trivial in $H^2(U;\R)$ then we can choose $\alpha$ so that
    %$$d_A\Phi=0,\quad F_A=0.$$
%\end{proposition}

%\begin{proof}
    %We have
    %$$F_A=-\frac12 \Phi[\Phi^*(\omega)-2d\alpha].$$
    %Since $\Phi^*(\omega)$ is exact, we can choose $\alpha$ making this vanish.
%\end{proof}

Next, recall from \cite[Section IV.1]{JaffeTaubes} the standard \emph{Bogomolnyi--Prasad--Sommerfield (BPS)} monopole $(\Phi_0,A_0)$ of degree $\mp1$ on $\mathbb{R}^3$. Identifying $\mathbb{R}^3\cong \{\R\bf{i}+\R\bf{j}+\R\bf{k}\}$ and recalling \cref{factors.2}, we have
$$\Phi_0(x):=\mp \left(\frac{1}{r \tanh(2r)}-\frac{1}{2r^2}\right)x,$$
and 
$$A_0(x):=\left(\frac{1}{r\sinh(2r)}-\frac{1}{2r^2}\right)\sum_{i=1}^3(x\times e_i)\,dx_i, $$ where $x \times e_i$ denotes the vector product, and $e_i$ the usual coordinate vectors. 
This pair satisfies the first order monopole equations
$$\ast d_{A_0} \Phi_0 = \pm F_{A_0},$$
as well as
$$E_1(\Phi_0,A_0)=4\pi.$$
As discussed on \cite[p.\ 105]{JaffeTaubes}, we then have the pointwise estimates
$$1-|\Phi_0(x)|=\frac{1}{2|x|}+O(e^{-4|x|})$$
and
$$|d_{A_0}\Phi_0|(x)=O(1/|x|^2).$$ 
In particular, we can write 
\begin{equation} \label{equation: PHI0}
\Phi_0(x)=\mp\frac{2r-\tanh(2r)}{2r\tanh(2r)}\Psi(x),     
\end{equation}
where $\Psi(x):=\frac{x}{|x|}$ is the singular unit section whose distributional Jacobian gives the Dirac mass $4\pi\delta_0$. This shows that $\Phi_0$ vanishes only at the origin and has degree $\mp1$ at infinity. Furthermore, thanks to \cite[p.\ 105]{JaffeTaubes}, we know that $(\Phi_0(x),A_0(x))$ is real analytic on $\mathbb{R}^3$. We summarize in the following lemma the scaling properties of the BPS monopole of degree $-1$.

\begin{lemma} \label{lemma: BPS monopole}
The rescalings $(\Phi^{BPS}_{\epsilon}(x),A^{BPS}_{\epsilon}(x))=(\Phi_0(x/\epsilon),\epsilon^{-1}A_0(x/\epsilon))$ of the BPS monopole satisfy
$$E_{\epsilon}(\Phi_{\epsilon},A_{\epsilon})=4\pi,$$
$$1-|\Phi_{\epsilon}|(x)=\frac{\epsilon}{2|x|}+O(e^{-4|x|/\epsilon}),$$
$$\epsilon|d_{A_{\epsilon}}\Phi_{\epsilon}|(x)=\epsilon^3|F_{A_{\epsilon}}|(x)=O(\epsilon^2/|x|^2),$$
and $\Phi_{\epsilon}(x)=|\Phi_{\epsilon}(x)|\frac{x}{|x|}$.
\end{lemma}

%Now, if $T\in \partial \mathcal{P}_{n-2}(M^n;\mathbb{Z})$ is a polyhedral $(n-3)$-boundary, \cite[Theorem 5.10]{ABOprescribed} (and a straightforward adaptation to the manifold setting) gives a map $u\in W^{1,3-\delta}(M)$ such that the distributional Jacobian $Ju=T,$ and $$|du(x)|\leq \frac{C}{\dist(\spt(T), \spt(S))}$$ where $S$ is a polyhedal $(n-4)$-current. 

\subsection{Proof of Theorem \ref{thm: recovery}: constructing a recovery sequence}\hfill\\
As before, let $M$ be a compact, oriented manifold with boundary, of dimension $n \geq 3$. By standard polyhedral approximation theorems (see in particular \cite[Theorem 4.2.22]{Federer}
and \cite[Proposition 4.2]{PPSgamma}), it is enough to prove Theorem \ref{thm: recovery} in the case where $T$ is a polyhedral $(n-3)$-current, and we may assume moreover that $T$ has multiplicity one on each face of $\spt(T)$ (see for instance \cite[Proposition 8.6]{ABO}). 

Now, let $P\in \mathcal{Z}_{n-3}(M,\partial M;\mathbb{Z})$ be a polyhedral $(n-3)$-boundary in the relative sense for $(M,\partial M)$ with multiplicity one on each face, with respect to some triangulation of $M$. In particular, we assume that there exists a polyhedral $(n-2)$-current $N\in \mathcal{P}_{n-2}(M; \mathbb{Z})$ such that
$$P-\partial N=Q,$$
where $\spt(Q)\subset \partial M$ and $P$ meets $\partial M$ transversally. As the starting point for our construction of a family of pairs $(\Phi_{\epsilon},A_{\epsilon})$ concentrating along $P$, we first recall the results of \cite[Section 5]{ABOprescribed}, and observe that they can be adapted to our setting.

\begin{proposition}\label{abo.prop}
    In the situation above, there exists a polyhedral set $S$ of dimension $\dim(S)\leq n-4$ and a locally Lipschitz map $v\in \Lip_{\loc}(M\setminus (P\cup S);S^2)$ such that
    $$*d(v^*(dA_{S^2}))=4\pi P$$
    and
    \begin{equation} \label{eq: bound on dv}
    |dv|(x)\leq \frac{C}{\dist(x,P\cup S)}    
    \end{equation}
for all $x\in M\setminus (P\cup S)$.
\end{proposition}
\begin{proof}
When $M$ is a domain in $\mathbb{R}^n$ and $\partial N=P$, this is precisely the content of \cite[Theorem 5.10]{ABOprescribed}. Moreover, those arguments carry over in a straightforward way to the case where $M$ is a closed Riemannian manifold and $P=\partial N$ for some polyhedral $(n-2)$-current $N$: indeed, note that all of the local constructions in \cite[Section 5]{ABOprescribed} can be implemented verbatim on small balls in $M$ whose geometry is approximately Euclidean.
%and the global inductive procedure in the proof of \cite[Theorem 5.6]{ABOprescribed} applies without modifications in this setting as well. 

More generally, let $M$ be a manifold with boundary, and let $P$ be a multiplicity-one polyhedral $(n-3)$-current meeting $\partial M$ transversally, such that $ \spt(P-\partial N)\subset \partial M$ for some polyhedral $(n-2)$-current $N$ in $M$. Note that we can replace the given metric $g$ on $M$ with another metric $g'$ satisfying $C^{-1}g\leq g'\leq Cg$ such that $(M,g')$ is isometric to $\partial M\times [0,\delta]$ on a tubular neighborhood of $\partial M$; then double $M$ across its boundary to get a closed manifold $\bar{M}$ carrying an isometric reflection $\rho: \bar{M}\to \bar{M}$ exchanging $M$ with its complement. Setting $\bar{N}=N+\rho_*N$ and recalling that $P$ meets $\partial M$ transversally, observe next that $\partial \bar{N}=P+\rho_*P=\bar{P}$, since $Q+\rho_*Q=0$. We can then apply the construction from \cite[Section 5]{ABOprescribed} to obtain a map $v\in \Lip_{loc}(\bar{M}\setminus (\bar{P}\cup \bar{S}); S^2)$ satisfying
$$*d(v^*(dA_{S^2}))=4\pi \bar{P}$$
and
$$|dv(x)|\leq \frac{C}{\dist(x,\bar{P}\cup \bar{S})}.$$
Restricting $v$ to $M\subset \bar{M}$ then gives the desired map.
\end{proof}

Now, fix an $(n-3)$-face $\Delta$ of $P$, and fix a small parameter $\delta>0$; as in \cite[Section 4]{PPSgamma}, we consider the subset $\Delta_{\delta}\subset \Delta$ consisting of points a distance $\geq \delta$ from the boundary and, for a constant $c>0$ to be chosen, identify the normal $c\delta$-tubular neighborhood $V_{\delta}(\Delta)$ with $\Delta_{\delta}\times B^3_{c\delta}$ by a normal exponential map which is almost an isometry for $\delta$ small.
We can fix $c$ so small that these neighborhoods $V_\delta(\Delta)$ have disjoint closures.

For $S$ as in Proposition \ref{abo.prop}, denote by $K$ the union of $S$ and the $(n-4)$-skeleton of $P$, and choose $C$ such that
$$B_{c\delta}(P)\setminus \bigcup_{\Delta}V_{\delta}(\Delta)\subset B_{C\delta}(K).$$
By a modification of the proof of \cite[Theorem 5.10]{ABOprescribed}, we can assume moreover that the map $v\in \Lip_{\loc}(M\setminus (P\cup S),S^2)$ agrees with the radial projection $(x,y)\to y/|y|$ on $V_{\delta}(\Delta).$

With this in mind, let $\rho(x):= \dist_{P\cup S}(x)$, and consider
$$\Phi_{\epsilon}(x):=\left(\frac{1}{\tanh(2\rho/\epsilon)}-\frac{\epsilon}{2\rho}\right)v.$$
Then, on each $V_{\delta}(\Delta)$, set
$$\tilde{A}_{\epsilon}(x):=\frac{1}{\epsilon}\left(\frac{1}{\sinh(2\rho/\epsilon)}-\frac{\epsilon}{2\rho}\right)\sum_{i=1}^3(v\times e_i)\,dx_i,$$
so that $(\Phi_{\epsilon},\tilde{A}_{\epsilon})$ matches the rescaled BPS monopole on the $B^3_{c\delta}$ factor of $V_{\delta}(\Delta)\cong \Delta_{\delta}\times B^3_{c\delta}$; and on $M\setminus (P\cup S)$, define $B_{\epsilon}:=\frac{1}{2}v^{-1}\,dv$ as in Proposition \ref{proposition: dAPhi = 0} with $\alpha = 0$, so that
$$d_{B_{\epsilon}}v=0$$
and
$$F_{B_{\epsilon}}=-\frac{1}{4}dv\wedge dv.$$
Next, let $\chi_K$ be a cut-off vanishing on $B_{C\delta}(K)$ and constantly equal to $1$ on $M\setminus B_{(C+1)\delta}(K)$, with $|d\chi_K|\leq C/\delta$ on the set $B_{(C+1)\delta}(K)$, which has volume $O(\delta^4)$. Let $\psi$ be another cut-off vanishing outside a $c\delta$-neighborhood of $P\cup S$ and constantly equal to $1$ on $B_{c\delta/2}(P\cup S)$. Note that $\tilde{A}_{\epsilon}$ is then well-defined on each component of $\spt(\psi)\cap \spt(\chi_K)$ and $B_{\epsilon}$ is well-defined on $\spt(\chi_K)\cap \spt(1-\psi)$. Defining
$$A_{\epsilon}:=\chi_K\cdot [\psi \tilde{A}_{\epsilon}+(1-\psi)B_{\epsilon}],$$
we claim that the following holds.

\begin{lemma} \label{lemma: Energy bound recovery}
Taking $\delta:=\epsilon^{3/4}$ in the preceding construction, the pair $(\Phi_{\epsilon},A_{\epsilon})$ above satisfies an estimate of the form
$$E_{\epsilon}(\Phi_{\epsilon},A_{\epsilon})\leq 4\pi \mathcal{H}^{n-3}(P)+C\epsilon^{1/32},$$
where $C$ is a constant depending on $P\cup S$. 
\end{lemma} 

Before beginning the proof, we record a few elementary estimates that will be of use during the proof.

\begin{lemma}\label{calc.lem}
    There is a constant $C>0$ such that, for any $t\in (0,\infty)$, we have
    \begin{equation}
        \left\vert \frac{1}{\tanh(t)}-\frac{1}{t} \right\vert \leq 1
    \end{equation}
    and
    \begin{equation}
        \left\vert \frac{1}{\sinh(t)}-\frac{1}{t} \right\vert \leq C\min\{t,1/t\},
    \end{equation}
    as well as 
    \begin{equation}
        \left\vert \frac{1}{\sinh(t)^2}-\frac{1}{t^2}\right\vert \leq C\min\{1,1/t^2\}.
    \end{equation}
\end{lemma}

\begin{proof}[Proof of Lemma \ref{lemma: Energy bound recovery}] 
We will show that most energy is contained in the region $\{\chi_K\equiv 1\}\cap\{\psi \equiv 1\}$, where the pair coincides with cylinders over rescaled BPS monopoles.

To begin, write
$$d_{A_{\epsilon}}\Phi_{\epsilon}=(1-\chi_K)d\Phi_{\epsilon}+\chi_K[\psi d_{\tilde{A}_{\epsilon}}\Phi_{\epsilon}+(1-\psi)d_{B_{\epsilon}}\Phi_{\epsilon}].$$
Since $d_{B_{\epsilon}}v=0$ on the support of $1-\psi$ by construction, we see that here
\begin{align*}
    |d_{B_{\epsilon}}\Phi_{\epsilon}|& = \left|d\left(\frac{1}{\tanh(2\rho/\epsilon)}-\frac{\epsilon}{2\rho}\right) \right| \leq \frac{2}{\epsilon} \left|\frac{\epsilon^2}{(2\rho)^2}-\frac{1}{\sinh(2\rho/\epsilon)^2}\right| \leq C\epsilon\rho^{-2},    
\end{align*}
and therefore
\begin{eqnarray*}
\frac{1}{\epsilon}\int_M \chi_K(1-\psi)|d_{B_{\epsilon}}\Phi_{\epsilon}|^2 \leq  C\epsilon\int_{M\setminus B_{\delta/2}(P\cup S)}\rho^{-4} \leq C\epsilon\delta^{-1}, 
\end{eqnarray*}
where $C$ is a constant depending on $P\cup S$. Moreover, we have
\begin{align*}
|d\Phi_{\epsilon}|&\leq \frac{C}{\epsilon}\left|\frac{\epsilon^2}{(2\rho)^2}-\frac{1}{\sinh(2\rho/\epsilon)^2}\right|+C \left|\frac{1}{\tanh(2\rho/\epsilon)}-\frac{\epsilon}{2\rho} \right|\rho^{-1}\\
&\leq \frac{C}{\epsilon}\min\{1,\epsilon^2/\rho^2\}+C/\rho\\
&\leq C\min\{1/\epsilon,1/\rho\},
\end{align*}
where we used \eqref{eq: bound on dv} to bound $\vert dv \vert$, and therefore
\begin{eqnarray*}
    \int_{M} \frac{1}{\epsilon}(1-\chi_K)|d\Phi_{\epsilon}|^2 \leq \int_{B_{(C+1)\delta}(K)} \frac{C}{\epsilon\rho^2} \leq C\frac{\delta^2}{\epsilon}, 
\end{eqnarray*}
where the last bound can be checked using the coarea formula and the fact that the level set
$B_{(C+1)\delta}(K)\cap\{\rho=t\}$ has area at most $C\delta t^2$.

Putting these estimates together, we see that
\begin{equation}\label{dir.est}
    \int_M\frac{1}{\epsilon}|d_{A_{\epsilon}}\Phi_{\epsilon}|^2\leq (1+C\sqrt{\epsilon \delta^{-1}+\epsilon^{-1}\delta^2})\frac{1}{\epsilon}\int_M(\chi_K\psi)^2|d_{\tilde{A}_{\epsilon}}\Phi_{\epsilon}|^2+C[\epsilon\delta^{-1}+\epsilon^{-1}\delta^2],
\end{equation}
giving the desired behavior for $\delta=\delta_{\epsilon}$ in the regime $\epsilon \ll \delta \ll \sqrt{\epsilon}$.

For the curvature estimate, first note that
\begin{equation} \label{equation: estimate for B and dB}
    |dB_{\epsilon}|+|B_{\epsilon}|^2\leq C|dv|+C|dv|^2\leq C\rho^{-2},
\end{equation}
where the last inequality follows from Proposition \ref{abo.prop}, and, after applying Lemma \ref{calc.lem}, 
\begin{equation} \label{equation: estimate for barA}
   |\tilde{A}_{\epsilon}|\leq \frac{C}{\epsilon}\min\{\rho/\epsilon, \epsilon/\rho\}. 
\end{equation}
Expanding $F_{A_{\epsilon}}$ pointwise and applying Cauchy-Schwarz gives an estimate of the form
\begin{align} \label{equation: estimate for FAeps}
\begin{aligned}
    |F_{A_{\epsilon}}| &\leq C|d\chi_K|(\psi|\tilde{A}_{\epsilon}|+(1-\psi)|B_{\epsilon}|)\\
    &\quad+C\chi_K|d\psi|(|\tilde{A}_{\epsilon}|+|B_{\epsilon}|)\\
    &\quad+C\chi_K(1-\psi)(|dB_{\epsilon}|+|B_{\epsilon}|^2)\\
    &\quad+\chi_K\psi(1-\psi)|\tilde A_\epsilon||B_\epsilon|+|\chi_K^2\psi^2-\chi_K\psi||\tilde A_\epsilon|^2+\chi_K\psi |F_{\tilde{A}_{\epsilon}}|;
\end{aligned}
\end{align} 
we will show that all terms but the last one on the right-hand side are $o(\epsilon^{-1/2})$ in $L^2$. We will analyze each term separately. Since $|d\chi_K|\leq C/\delta$ and $d\chi_K$ is supported on a $(C+1)\delta$-neighborhood of $K$, we use the preceding estimates for $B_{\epsilon}$ and $\tilde{A}_{\epsilon}$ to check directly that
\begin{align*}
    \int_M \epsilon |d\chi_K|^2(\psi^2|\tilde{A}_{\epsilon}|^2+(1-\psi)^2|B_{\epsilon}|^2) \leq &\frac{C\epsilon}{\delta^2}\int_{B_{(C+1)\delta}(K)}(\epsilon^{-2}\min\{\rho/\epsilon,\epsilon/\rho\}^2+C\delta^{-2})\\
    \leq & C\frac{\epsilon}{\delta}+C\epsilon\leq C\frac{\epsilon}{\delta}, 
\end{align*} 
where the first inequality follows form \eqref{equation: estimate for B and dB} and \eqref{equation: estimate for barA}.

Next, note that $|d\psi|$ is bounded by $\frac{C}{\delta}$ and supported in the annular region $c\delta/2\leq \rho \leq c\delta$, so that
$$|d\psi|^2(|\tilde{A}_{\epsilon}|^2+|B_{\epsilon}|^2)\leq \frac{C}{\delta^2}\delta^{-2}=C\delta^{-4},$$ where we used again the bounds \eqref{equation: estimate for B and dB} and \eqref{equation: estimate for barA}. Integrating then over $\spt(d\psi)$, which has volume of the order $\delta^3$, gives
$$\epsilon \int_{M \cap \spt(d\psi)} |d\psi|^2(|\tilde{A}_{\epsilon}|^2+|B_{\epsilon}|^2)\leq \frac{C\epsilon}{\delta}.$$
For the third line of \eqref{equation: estimate for FAeps}, we have
$$\chi_K(1-\psi)(|dB_{\epsilon}|+|B_{\epsilon}|^2)\leq C(1-\psi)\rho^{-2}$$
and, since $1-\psi$ is supported on $\{\rho\geq c\delta/2\}$, integration gives
$$\epsilon \int_M \chi_K^2(1-\psi)^2(|dB_{\epsilon}|+|B_{\epsilon}|^2)^2\leq C\epsilon \int_{\{\rho\geq c\delta/2\}}\rho^{-4}\leq C\frac{\epsilon}{\delta}.$$
Finally, note that
$$|\tilde{A}_{\epsilon}||B_\epsilon|+|\tilde A_\epsilon|^2\leq \frac{C}{\rho^2}\quad\text{on }\{\chi_K>0,\ \psi\neq1\},$$
so similarly, multiplying by $\epsilon$ and integrating, we get as before
$$\epsilon\int_{\{\psi\neq1\}} [\chi_K\psi(1-\psi)|\tilde A_\epsilon||B_\epsilon|]^2\leq \frac{C\epsilon}{\delta},$$
while
$$\epsilon\int_{\{\psi=1\}}[|\chi_K^2\psi^2-\chi_K\psi||\tilde A_\epsilon|^2]^2
\le \epsilon\int_{B_{(C+1)\delta}(K)}\epsilon^{-4}\min\{\rho/\epsilon,\epsilon/\rho\}^4
\le C\frac{\delta^4}{\epsilon^3}\cdot\frac{\delta^4}{\epsilon^4}
= C\frac{\delta^8}{\epsilon^7}. $$

Putting the preceding estimates together, we deduce that
\begin{equation}\label{f.est}\epsilon \int_M |F_{A_{\epsilon}}|^2\leq (1+C\sqrt{\epsilon/\delta+\delta^8/\epsilon^7})\epsilon \int_M \chi_K^2\psi^2 |F_{\tilde{A}_{\epsilon}}|^2+C(\epsilon/\delta+\delta^8/\epsilon^7).
\end{equation}
In particular, taking $\delta:=\epsilon^{15/16}$ and combining this estimate with \eqref{dir.est}, we see that
$$E_{\epsilon}(\Phi_{\epsilon},A_{\epsilon})\leq (1+C\epsilon^{1/32})\int_M\chi_K^2\psi^2\left(\frac{1}{\epsilon}|d_{\tilde{A}_{\epsilon}} \Phi_\epsilon|^2+\epsilon|F_{\tilde{A}_{\epsilon}}|^2\right)+C\epsilon^{1/16}.$$
By construction, the pair $(\Phi_{\epsilon},\tilde{A}_{\epsilon})$ coincides on $V_{\delta}(\Delta)$ with the product of the $\epsilon$-rescaled BPS monopole on $\mathbb{R}^3$ and $\Delta$, so we see that
$$\int_M \chi_K^2\psi^2\left(\frac{1}{\epsilon}|d_{\tilde{A}_{\epsilon}} \Phi_\epsilon|^2+\epsilon|F_{\tilde{A}_{\epsilon}}|^2\right)\leq 4\pi\sum_{\Delta}\mathcal{H}^{n-3}(\Delta)+C\delta, $$ after applying Lemma \ref{lemma: BPS monopole}, where the last error term comes from the fact that $V_\delta(\Delta)$
is almost isometric to a product. Hence,
$$E_{\epsilon}(\Phi_{\epsilon},A_{\epsilon})\leq 4\pi\mathcal{H}^{n-3}(P)+C\epsilon^{1/32},$$
as claimed.
\end{proof}

\begin{remark}\label{onephi.est}
Moreover, note that the section $\Phi_{\epsilon}$ satisfies $|\Phi_\epsilon|\le 1$ and
$$1-|\Phi_{\epsilon}|=\frac{\epsilon}{2\rho}-\frac{e^{-2\rho/\epsilon}}{\sinh(2\rho/\epsilon)}\le\frac{\epsilon}{2\rho},$$
%so that $1-|\Phi_{\epsilon}|\leq C$ on $B_{\epsilon}(P\cup S)$ while $1-|\Phi_{\epsilon}|\leq \frac{\epsilon}{\rho}$ globally,
so that a simple application of the coarea formula yields the integral estimates
$$\int_M (1-|\Phi_{\epsilon}|)\leq C\epsilon$$
and
\begin{equation}\label{tozero}\int_M (1-|\Phi_{\epsilon}|)^2\leq C\epsilon^2,\end{equation}
and in fact
$$\lim_{r\to0}\limsup_{\epsilon\to0}\int_{\{\rho<r\}} \frac{(1-|\Phi_{\epsilon}|)^2}{\epsilon^2}\le C_0\mathcal{H}^{n-3}(P)$$
for an absolute constant $C_0$. Given $\eta>0$ small, let $\varphi_\eta:\R\to\R$ smooth
with $\varphi_\eta(t)=1$ for $t\le 1-\eta$, $\varphi_\eta(t)=1/t$ for $t\ge 1-\eta^2$, and $\varphi_\eta\le1+2\eta^2$, as well as $|\varphi_\eta'|\le 2\eta$.
Then, taking
$$\tilde\Phi_{\epsilon}:=\varphi_\eta(|\Phi_\epsilon|)\Phi_\epsilon,$$
it is easy to check that $E_\epsilon(\tilde\Phi_\epsilon,A_\epsilon)\le(1+C\eta)E_\epsilon(\Phi_\epsilon,A_\epsilon),$
as well as
$$\limsup_{\epsilon\to0}\int_{M} \frac{(1-|\tilde\Phi_{\epsilon}|)^2}{\epsilon^2}\le C_0'\mathcal{H}^{n-3}(P),$$
since for any fixed $r>0$ we have $1-|\Phi_\epsilon|\le\frac{\epsilon}{2r}$ on $\{\rho\ge r\}$,
and thus $1-|\tilde\Phi_\epsilon|=0$ here, once $\epsilon\le2\eta^2r$. Finally, using \eqref{tozero},
we can see that $Z(\tilde\Phi_\epsilon,A_\epsilon)$ has the same limit as $Z(\Phi_\epsilon,A_\epsilon)$,
by an argument entirely analogous to the one used to show \eqref{same.lim}.
Hence, by a standard diagonal argument, up to replacing $(\Phi_\epsilon,A_\epsilon)$
with a new pair, we can also guarantee that
$$\limsup_{\epsilon\to0}\int_{M} \frac{(1-|\tilde\Phi_{\epsilon}|)^2}{\epsilon^2}\le C_0'\mathcal{H}^{n-3}(P).$$
%indeed, as $\epsilon\to 0$, we have an estimate of the form
%$$\lim_{\epsilon\to 0}\int_M\frac{(1-|\Phi_{\epsilon}|)^2}{\epsilon^2}\leq C_0\mathcal{H}^{n-3}(P),$$
%with $C_0$ independent of $P$ and $S$.
\end{remark}

We can now complete the proof of Theorem \ref{thm: recovery} in a few lines.

\begin{proof}[Proof of Theorem \ref{thm: recovery}]
As discussed before, given a relative integral $(n-3)$-boundary $T$ in $M$, the results of \cite[Section 4.2]{Federer} imply the existence of a sequence of polyhedral relative $(n-3)$-boundaries $P_j$ in $M$ such that $P_j\to T$ in the flat topology and
$$\lim_{j\to\infty}\mathbb{M}(P_j)=\mathbb{M}(T).$$
Without loss of generality, we can assume moreover that $P_j$ has multiplicity one on each $(n-3)$-face, again by, for instance, \cite[Proposition 8.6]{ABO}. Then, for each $P_j$, Lemma \ref{lemma: Energy bound recovery} and Remark \ref{onephi.est} supply for each $\epsilon\in (0,\epsilon_j)$ a pair $(\Phi^j_{\epsilon},A^j_{\epsilon})$ satisfying
$$E_{\epsilon}(\Phi_{\epsilon}^j,A_{\epsilon}^j)\leq 4\pi \mathbb{M}(P_j)+\frac{1}{j},$$
$$\mathbb{F}(4\pi P_j-Z(\Phi_{\epsilon}^j,A_{\epsilon}^j))\leq \frac{1}{j},$$
and
$$\int_M\frac{(1-|\Phi_{\epsilon}^j|)^2}{\epsilon^2}\leq C\mathbb{M}(P_j).$$
Setting $(\Phi_{\epsilon},A_{\epsilon}):=(\Phi_{\epsilon}^j,A_{\epsilon}^j)$ for $\epsilon \in [\epsilon_{j+1},\epsilon_j)$, it then follows that the family $(\Phi_{\epsilon},A_{\epsilon})$ satisfies the conclusions of Theorem \ref{thm: recovery}.

\end{proof}

\section{Approximating the Plateau problem}\label{plat.sec}

In this section, we conclude with the proof of Theorem \ref{plateau.thm}. Throughout, we assume without loss of generality that $M\subset N$ is a domain in the interior of a larger compact $n$-manifold $N$, such that the nearest-point projection $P:V\to \partial M$ is well-defined and smooth on the complement $V:=N\setminus M$. In this setting, we observe next that pairs $(\Phi,A)$ on $M$ and $V$ with matching Dirichlet data can be glued continuously across $\partial M$ after a change of gauge.

\begin{lemma}\label{paste.lem}
Given a pair $\Phi\in C^{\infty}(M;\mathfrak{su}(2))$ and $A\in \Omega^1(M;\mathfrak{su}(2))$ on $M$, and a pair $\Phi'\in C^{\infty}(V;\mathfrak{su}(2))$ and $A'\in \Omega^1(V;\mathfrak{su}(2))$ on $V$ such that $$\iota_{\partial M\hookrightarrow M}^*(\Phi,A)=\iota_{\partial M\hookrightarrow V}^*(\Phi',A'),$$
there exists a Lipschitz continuous pair $(\hat{\Phi},\hat{A})$ on $N$ such that $(\hat{\Phi},\hat{A})$ agrees with $(\Phi,A)$ on $M$ and $(\Phi',A')$ on $V$ after changes of gauge.
\end{lemma}
\begin{proof}
Since $\iota_{\partial M}^*(\Phi,A)=\iota_{\partial M}^*(\Phi',A')$ by assumption, the result would be immediate, with no gauge changes required, if we also had $A(\nu)=A'(\nu)$ along $\partial M$, where $\nu$ is the outward unit normal to $\partial M$. We claim that this can always be arranged after a change of gauge. In particular, we can find smooth maps $g: M\to \SU(2)$ and $h: V\to \SU(2)$ with $g=h=1$ on $\partial M$, such that the gauge-transformed pairs $(\Phi^g,A^g)$ and $((\Phi')^h,(A')^h)$ satisfy $A^g(\nu)=(A')^h(\nu)=0$ on $\partial M$, and can therefore be patched together to give a Lipschitz pair $(\hat{\Phi},\hat{A})$ on $N$.

We briefly explain how to construct the desired map $g: M\to \SU(2)$; the construction of $h: V\to \SU(2)$ is identical. Let $U\subset M$ be a tubular neighborhood of $\partial M$ on which the nearest-point projection $P: U\to\partial M$ is well-defined and smooth, and fix a cut-off function $\chi\in C_c^{\infty}(U)$ such that $\chi(x)=\dist_{\partial M}(x)$ on a smaller tubular neighborhood of $\partial M$. In the reference gauge, let $\phi: \partial M\to \mathfrak{su}(2)$ be given by $\phi(x):=-\langle A(x),\nu\rangle$, and define $g: M\to \SU(2)$ by the matrix exponential
    $$g(x):=\exp(\chi(x)\phi(P(x))).$$
Then $g$ satisfies $g=1$ on $\partial M\cup (M\setminus U)$, and 
    $$-\frac{\partial g}{\partial \nu}(x)=\left.\frac{\partial}{\partial r}e^{r\phi(P(x))}\right|_{r=0}=\phi(x).$$
In particular, it follows that, on $\partial M$,
    $$\frac{\partial g^{-1}}{\partial \nu}=-\frac{\partial g}{\partial\nu}=-\langle A(x),\nu\rangle,$$
and so, recalling throughout that $g=1$ on $\partial M$, we have
$$\langle A^g,\nu\rangle=\langle A+dg^{-1},\nu\rangle=0,$$
as desired.
\end{proof}

Now, fix $\Gamma^{n - 4} \subset \partial M$ to be a smooth $(n - 4)$-dimensional submanifold of $\partial M$ such that $[\Gamma] = 0 \in H_{n - 4}(M; \mathbb{Z})$. For technical reasons, it is useful to introduce the submanifold $\hat{\Gamma}$ in $V$ given by
$$\hat{\Gamma}:=P^{-1}(\Gamma),$$
where $P: V\to \partial M$ is the nearest-point projection. 

In particular, observe that if $S\in \mathcal{I}_{n-3}(M;\mathbb{Z})$ is any integral $(n-3)$-current such that $\partial S=\Gamma$ and $[S]=0\in H_{n-3}(M,\partial M;\mathbb{Z})$, then
$$\hat{S}:=S+\hat{\Gamma}\in \mathcal{I}_{n-3}(N;\mathbb{Z})$$
defines a relative $(n-3)$-boundary in the larger manifold $N$, for a suitable orientation of $\hat\Gamma$. In particular, we can apply Theorem \ref{thm: recovery} to the current $\hat{S}$ in $N$, to obtain the following.

\begin{lemma} \label{lemma: boundary approximation plateau}
    For any integral $(n-3)$-current $S$ in $M$ with $\partial S=\Gamma$, there exists a family of pairs $\Phi^S_{\epsilon}: N \to \mathfrak{su}(2)$ and $A^S_{\epsilon}\in \Omega^1(N;\mathfrak{su}(2))$ such that $|\Phi_{\epsilon}^S|\leq 1$,  
    $$Z(\Phi^S_{\epsilon},A^S_{\epsilon})\rightharpoonup^* 4\pi \hat{S},$$
    $$\limsup_{\epsilon\to 0}\int_M\frac{(1-|\Phi^S_{\epsilon}|)^2}{\epsilon^2}\leq C\mathbb{M}(\hat{S}),$$
    and
    $$\lim_{\epsilon\to 0}E_{\epsilon}(\Phi_{\epsilon}^S,A_{\epsilon}^S)=4\pi \mathbb{M}(\hat{S}).$$
\end{lemma}
Moreover, if $S$ agrees with $P^{-1}(\Gamma)$ near $\partial M$ in Lemma \ref{lemma: boundary approximation plateau}, thus meeting $M$ transversally, then the pairs $(\Phi_{\epsilon}^S,A_{\epsilon}^S)$ have no energy concentration along $\partial M$, and it follows in particular that
\begin{equation}\label{trv.cond}
\lim_{\epsilon\to 0}\|Z(\Phi_{\epsilon}^S,A_{\epsilon}^S)-4\pi \hat{\Gamma}\|_{C^1(V)^*}=0
\end{equation}
and
\begin{equation}\label{m.ener.cond}
\lim_{\epsilon\to 0}E_{\epsilon}(\Phi_{\epsilon}^S,A_{\epsilon}^S; M)=4\pi \mathbb{M}(S). 
\end{equation}
Now, consider a sequence $S_j\in \mathcal{I}_{n-3}(M;\mathbb{Z})$ satisfying $\partial S_j=\Gamma$ and
$$\mathbb{M}(S_j)< \inf\{\mathbb{M}(S)\mid S\in \mathcal{I}_{n-3}(M),\text{ }\partial S=\Gamma\}+\frac{1}{j}.$$
Without loss of generality, we can assume moreover that each $S_j$ meets $\partial M$ transversally as above, though of course we cannot enforce this in a uniform way as $j\to\infty$ without some convexity assumption on $M$. Applying Lemma \ref{lemma: boundary approximation plateau} to the sequence $S_j$, it follows that there exist $\epsilon_j>0$ and pairs $(\Phi_{\epsilon}^j,A_{\epsilon}^j)$ on $N$ such that for $\epsilon<\epsilon_j$
\begin{equation}\label{ener.glob}
    E_{\epsilon}(\Phi_{\epsilon}^j,A_{\epsilon}^j)+\int_N\frac{(1-|\Phi_{\epsilon}^j|)^2}{\epsilon^2}\leq C \mathbb{M}(\hat{S}_j)
\end{equation}
and, by \eqref{trv.cond} and \eqref{m.ener.cond},
\begin{equation}
    \|Z(\Phi_{\epsilon}^j,A_{\epsilon}^j)\mres V-4\pi \hat{\Gamma}\|_{C^1(N)^*}<\frac{1}{j},
\end{equation}
as well as
\begin{equation}\label{ener.m}
    E_{\epsilon}(\Phi_{\epsilon}^j,A_{\epsilon}^j; M)\leq 4\pi\inf\{\mathbb{M}(S)\mid S\in \mathcal{I}_{n-3}(M),\text{ }\partial S=\Gamma\}+\frac{4\pi}{j}.
\end{equation}
Replacing $\epsilon_j$ with $\min\{\epsilon_1,\ldots,\epsilon_j,\frac{1}{j}\}$ if necessary, we can assume moreover that $\epsilon_j$ is a decreasing sequence with $\lim_{j\to\infty}\epsilon_j=0$. We then define the pairs $\Psi_{\epsilon}\in C^{\infty}(\partial M,\mathfrak{su}(2))$, $B_{\epsilon}\in\Omega^1(\partial M;\mathfrak{su}(2))$ by
\begin{equation}\label{dir.def}
(\Psi_{\epsilon},B_{\epsilon}):=\iota_{\partial M}^*(\Phi_{\epsilon}^j,A_{\epsilon}^j)\text{ for }\epsilon \in (\epsilon_j,\epsilon_{j+1}].
\end{equation}
We then assert that the conclusion of Theorem \ref{plateau.thm} holds for this choice of $(\Psi_{\epsilon},B_{\epsilon}).$

\begin{proposition}\label{plat.prop}
    For $(\Psi_{\epsilon},B_{\epsilon})$ given by \eqref{dir.def} above, setting
   $$\alpha_{\epsilon}(\Psi_\epsilon, B_\epsilon):=\inf\{E_{\epsilon}(\Phi,A)\mid \iota_{\partial M}^*(\Phi,A)=(\Psi_{\epsilon},B_{\epsilon})\},$$
    we have
    $$\lim_{\epsilon\to 0}\alpha_{\epsilon}(\Psi_\epsilon, B_\epsilon)=4\pi \inf\{\mathbb{M}(T)\mid T\in \mathcal{I}_{n-3}(M;\mathbb{Z})\text{ with }\partial T=\Gamma\},$$
    and for any family of pairs $(\Phi_{\epsilon},A_{\epsilon})$ on $M$ with 
    \begin{equation}\label{near.min.fam}
    \iota_{\partial M}^*(\Phi_{\epsilon},A_{\epsilon})=(\Psi_{\epsilon},B_{\epsilon})\text{ and }E_{\epsilon}(\Phi_{\epsilon},A_{\epsilon})\leq \alpha_{\epsilon}(\Psi_\epsilon, B_\epsilon)+o(1),
    \end{equation}
    there is a mass-minimizing extension $T\in \mathcal{I}_{n-3}(M;\mathbb{Z})$ with $\partial T=\Gamma$ such that $Z(\Phi_{\epsilon},A_{\epsilon})\rightharpoonup^* 4\pi T$ along a subsequence.
\end{proposition}

\begin{proof}
By construction, we note that the pairs $(\Psi_{\epsilon},B_{\epsilon})$ admit smooth extensions $(\Phi_{\epsilon}',A_{\epsilon}')=(\Phi_{\epsilon}^j,A_{\epsilon}^j)$ to $N=M\cup V$ satisfying
\begin{equation}\label{nglob}
    E_{\epsilon}(\Phi_{\epsilon}',A_{\epsilon}')+\int_N\frac{(1-|\Phi_{\epsilon}'|)^2}{\epsilon^2}\leq C,
\end{equation}
\begin{equation}\label{vlim}
    \lim_{\epsilon \to 0}\|Z(\Phi_{\epsilon}',A_{\epsilon}')\mres V-4\pi \hat{\Gamma}\|_{C^1(N)^*}=0
\end{equation}
and
\begin{equation*}
    \limsup_{\epsilon\to 0}E_{\epsilon}(\Phi_{\epsilon}',A_{\epsilon}'; M)\leq 4\pi\inf\{\mathbb{M}(S)\mid S\in \mathcal{I}_{n-3}(M),\text{ }\partial S=\Gamma\}.
\end{equation*}
Since $\alpha_{\epsilon}(\Psi_\epsilon, B_\epsilon)\leq E_{\epsilon}(\Phi_{\epsilon}',A_{\epsilon}';M)$ by definition, it follows immediately that
$$\limsup_{\epsilon\to 0}\alpha_{\epsilon}(\Psi_\epsilon, B_\epsilon)\leq 4\pi \inf\{\mathbb{M}(T)\mid T\in \mathcal{I}_{n-3}(M;\mathbb{Z})\text{ with }\partial T=\Gamma\}.$$

To complete the proof, it therefore suffices to show that any family $(\Phi_{\epsilon},A_{\epsilon})$ in $M$ satisfying \eqref{near.min.fam} has $Z(\Phi_{\epsilon},A_{\epsilon})\rightharpoonup^* 4\pi T$ subsequentially, where $T$ is an integral $(n-3)$-current in $M$ with $\partial T=\Gamma$; indeed, it then follows from the a priori mass bound
\begin{align*}
    4\pi\mathbb{M}(T)&\leq \liminf_{\epsilon \to 0}E_{\epsilon}(\Phi_{\epsilon},A_{\epsilon})\\
    &\leq \liminf_{\epsilon\to 0}\alpha_{\epsilon}(\Psi_\epsilon, B_\epsilon)\\
    &\leq 4\pi \inf\{\mathbb{M}(S)\mid S\in \mathcal{I}_{n-3}(M;\mathbb{Z})\text{ with }\partial S=\Gamma\}
\end{align*}
that $T$ is a mass-minimizing fill-in for $\Gamma$, of mass $\lim_{\epsilon\to 0}\alpha_{\epsilon}(\Psi_\epsilon, B_\epsilon)$.

So, let $(\Phi_{\epsilon},A_{\epsilon})$ be any family satisfying \eqref{near.min.fam}, and suppose additionally that
\begin{equation}\label{phinear1}
\lim_{\epsilon\to 0}\int_M\frac{(1-|\Phi_{\epsilon}|)^2}{\epsilon}=0.
\end{equation}
Applying Lemma \ref{paste.lem} to the pairs $(\Phi_{\epsilon}',A_{\epsilon}')$ on $V$ and $(\Phi_{\epsilon},A_{\epsilon})$ on $M$, it follows from \eqref{nglob} and \eqref{vlim} that $(\Phi_{\epsilon},A_{\epsilon})$ admits a Lipschitz extension $(\hat{\Phi}_{\epsilon},\hat{A}_{\epsilon})$ to the larger manifold $N$ satisfying the hypotheses of Theorem \ref{thm: main theorem} on $N$, as well as 
\begin{equation}\label{ext.z.lim}
    \lim_{\epsilon\to 0}\|Z(\hat{\Phi}_{\epsilon},\hat{A}_{\epsilon})\mres V-4\pi \hat{\Gamma}\|_{C^1(N)^*}=0.
\end{equation}
Applying Theorem \ref{thm: main theorem}, it then follows that along a subsequence $Z(\hat{\Phi}_{\epsilon},\hat{A}_{\epsilon})$ converges to an integral $(n-3)$-current $4\pi\hat{T}$ in $N$ with $\spt(\partial T)\subseteq \partial N$. On the other hand, writing
$$Z(\hat{\Phi}_{\epsilon},\hat{A}_{\epsilon})=Z(\Phi_{\epsilon},A_{\epsilon})\mres M+Z(\hat{\Phi}_{\epsilon},\hat{A}_{\epsilon})\mres V,$$
it then follows from \eqref{ext.z.lim} that
$$4\pi T:=\lim_{\epsilon\to 0}Z(\Phi_{\epsilon},A_{\epsilon})=4\pi (\hat{T}-\hat{\Gamma}).$$
In particular, we deduce that $T$ is globally an integral $(n-3)$-current on $M$ satisfying $\partial T=\Gamma,$ as desired.

To complete the proof, it remains to show that any family satisfying \eqref{near.min.fam} also satisfies \eqref{phinear1}. To see this, first consider the case where $\Phi_{\epsilon}$ minimizes $\Phi\mapsto \int_M|d_{A_{\epsilon}}\Phi|^2$ with respect to the boundary condition $\Phi|_{\partial M}=\Psi_{\epsilon}$. Note that, since $A_{\epsilon}$ is fixed and smooth, this minimization problem is well-posed, and yields a unique section $\Phi_{\epsilon}$ satisfying
$$d_{A_{\epsilon}}^*d_{A_{\epsilon}}\Phi_{\epsilon}=0\quad\text{in }M,$$
which has $|\Phi_\epsilon|\le1$ since capping $|\Phi_\epsilon|$
at $1$ can only decrease the energy.
As a consequence, it follows that 
\begin{equation} \label{pde feps}
d^*d (1-|\Phi_{\epsilon}|^2)=2|d_{A_{\epsilon}}\Phi_{\epsilon}|^2   
\end{equation}
and, setting
$$f_{\epsilon}:=\frac{1-|\Phi_{\epsilon}|^2}{\sqrt{\epsilon}},$$
we find that
$$\|df_{\epsilon}\|_{L^2(M)}^2\leq \frac{C}{\epsilon}\int_M|d_{A_{\epsilon}}\Phi_{\epsilon}|^2\leq C$$
and
\begin{equation}\label{lap.small}
\|d^*df_{\epsilon}\|_{L^1(M)}\leq \frac{C}{\sqrt{\epsilon}}\int_M|d_{A_{\epsilon}}\Phi_{\epsilon}|^2\leq C\sqrt{\epsilon}.
\end{equation}
Moreover, recalling \eqref{nglob} and writing 
$$\xi_{\epsilon}:=\frac{1-|\Phi_{\epsilon}'|^2}{\sqrt{\epsilon}},$$
we see that $\xi_{\epsilon} - f_{\epsilon}=0$ on $\partial M$ and $\int_M|d\xi_{\epsilon}|^2\leq CE_{\epsilon}(\Phi_{\epsilon}',A_{\epsilon}')\leq C$, while \eqref{nglob} gives
$$\|\xi_{\epsilon}\|_{L^2(M)}^2\leq C\epsilon.$$
Now, applying the Rellich compactness theorem to $f_{\epsilon}$ and $f_{\epsilon}-\xi_{\epsilon}$, we see that, along subsequences, $f_{\epsilon}$ has a strong $L^2$ limit
$$f=\lim_{\epsilon\to 0}f_{\epsilon},$$
which coincides with the $L^2$ limit of $f_{\epsilon}-\xi_{\epsilon}\in W_0^{1,2}(M)$ thanks to the estimate that we have for $\Vert \xi_\epsilon \Vert_{L^2(M)}$, so in particular $f=0$ on $\partial M$. On the other hand, it follows from \eqref{lap.small} that $d^*df=0$ in $M$, which together with the vanishing of $f$ on $\partial M$ forces $f\equiv 0$ on $M$. In other words, $\epsilon^{-1/2} (1-|\Phi_{\epsilon}|^2)$ converges strongly to $0$ in $L^2(M)$, which is the same as \eqref{phinear1} since $|\Phi_\epsilon|\le1$.
%in this case, since a simple maximum principle argument together with \eqref{pde feps} implies $\vert \Phi_{\epsilon} \vert \leq \max \vert \Psi_{\epsilon} \vert\leq 1$, hence $1 - \vert \Phi_\epsilon \vert \leq 1 - \vert \Phi_\epsilon \vert^2$. 

Now, for any family $(\Phi_{\epsilon},A_{\epsilon})$ satisfying \eqref{near.min.fam}, let $(\Phi_{\epsilon}^{min},A_{\epsilon})$ be the family obtained by replacing $\Phi_{\epsilon}$ with the minimizer for $\Phi\mapsto \int_M|d_{A_{\epsilon}}\Phi|^2$ with the same boundary data. Then we see that
$$\alpha_{\epsilon}(\Psi_\epsilon, B_\epsilon)\leq E_{\epsilon}(\Phi_{\epsilon}^{min},A_{\epsilon})\leq E_{\epsilon}(\Phi_{\epsilon},A_{\epsilon})\leq \alpha_{\epsilon}(\Psi_\epsilon, B_\epsilon) + o(1),$$
and since $d_{A_{\epsilon}}^*d_{A_{\epsilon}}\Phi_{\epsilon}^{min}=0$ while $\Phi_{\epsilon}-\Phi_{\epsilon}^{min}=0$ on $\partial M$, we deduce that
\begin{align*}
    \frac{1}{\epsilon}\int_M|d_{A_{\epsilon}}(\Phi_{\epsilon}-\Phi_{\epsilon}^{min})|^2
    &=\frac{1}{\epsilon}\int_M |d_{A_{\epsilon}}\Phi_{\epsilon}|^2-|d_{A_{\epsilon}}\Phi_{\epsilon}^{min}|^2\\
    &=E_{\epsilon}(\Phi_{\epsilon},A_{\epsilon})-E_{\epsilon}(\Phi_{\epsilon}^{min},A_{\epsilon})\\
    &\to 0
\end{align*}
as $\epsilon\to 0$. In particular, since $\Phi_{\epsilon}^{min}=\Phi_{\epsilon}$ on $\partial M$ and $|d|\Phi_{\epsilon}^{min}-\Phi_{\epsilon}||\leq |d_{A_{\epsilon}}(\Phi_{\epsilon}^{min}-\Phi_{\epsilon})|$, it follows that
$$\frac{|\Phi_{\epsilon}-\Phi_{\epsilon}^{min}|}{\sqrt{\epsilon}}\to 0$$
in $W^{1,2}(M)$ as $\epsilon \to 0$, and consequently
$$\lim_{\epsilon\to 0}\frac{1-|\Phi_{\epsilon}|}{\sqrt{\epsilon}}=\lim_{\epsilon\to 0}\frac{1-|\Phi_{\epsilon}^{min}|}{\sqrt{\epsilon}}=0$$
in $L^2(M)$. This confirms \eqref{phinear1} for any family satisfying \eqref{near.min.fam}, completing the proof.
\end{proof}

\printbibliography
\end{document}